\documentclass[11pt]{amsart}

\usepackage{fullpage}

\usepackage{amsmath, amscd, amsthm, amssymb}
\usepackage{enumitem}

\def\C{{\mathbf C}}

\def\A{{\mathbf A}}

\newtheorem{theorem}{Theorem}

\newtheorem{lemma}[theorem]{Lemma}

\newtheorem{proposition}[theorem]{Proposition}

\newtheorem{claim}[theorem]{Claim}

\theoremstyle{definition}
\newtheorem{definition}[theorem]{Definition}

\theoremstyle{remark}
\newtheorem{remark}[theorem]{Remark}

\newcommand{\mm}[4]{\left(\begin{smallmatrix} #1 & #2\\ #3 & #4\end{smallmatrix}\right)}

\DeclareMathOperator{\val}{val}

\DeclareMathOperator{\SO}{SO}

\DeclareMathOperator{\Spin}{Spin}
\DeclareMathOperator{\GSpin}{GSpin}

\DeclareMathOperator{\GSp}{GSp}

\DeclareMathOperator{\GU}{GU}

\DeclareMathOperator{\SL}{SL}
\DeclareMathOperator{\GL}{GL}

\DeclareMathOperator{\Clif}{Clif}
\DeclareMathOperator{\charf}{char}
\DeclareMathOperator{\diag}{diag}
\DeclareMathOperator{\Stab}{Stab}

\begin{document}
\title{Unramified Godement-Jacquet theory for the spin similitude group}
\author{Aaron Pollack}
\address{Department of Mathematics\\ Institute for Advanced Study\\ Princeton, NJ USA}
\email{aaronjp@math.ias.edu}
\thanks{This work partially supported by the NSF through grant DMS-1401858.}

\begin{abstract} Suppose $F$ is a non-archimedean local field.  The classical Godement-Jacquet theory is that one can use Schwartz-Bruhat functions on $n \times n$ matrices $M_n(F)$ to define the local standard $L$-functions on $\GL_n$.  The purpose of this partly expository note is to give evidence that there is an analogous and useful ``approximate" Godement-Jacquet theory for the standard $L$-functions on the special orthogonal groups $\SO(V)$: One replaces $\GL_n(F)$ with $\GSpin(V)(F)$ and $M_n(F)$ with $\Clif(V)(F)$, the Clifford algebra of $V$.  More precisely, we explain how a few different local unramified calculations for standard $L$-functions on $\SO(V)$ can be done easily using Schwartz-Bruhat functions on $\Clif(V)(F)$.  We do not attempt any of the ramified or global theory of $L$-functions on $\SO(V)$ using Schwartz-Bruhat functions on $\Clif(V)$.  \end{abstract}

\maketitle

\section{Introduction}
Following Braverman-Kazhdan \cite{bk1, bk2, bk3}, a topic of interest in automorphic forms is extending the classical Godement-Jacquet theory for the standard $L$-function on $\GL_n$ to other reductive groups and $L$-functions.  Such a program has received a lot of attention recently, and from many different perspectives, for instance in \cite{lafforgue, ngo, bns, sakellaridis3, li_FE, getz, casselman}, to name only a few.  Recall that, if $F$ is a $p$-adic local field, $\pi$ is an unramified irreducible representation of $\GL_n(F)$, and $v_0, v_0^\vee$ are spherical vectors in $\pi$ and its contragredient normalized so that $\langle v_0^\vee, v_0\rangle =1$, the so-called Godement-Jacquet \cite{tamagawa,GJ} integral for the Standard $L$-function of $\pi$ is
\begin{equation}\label{GJ1} L(\pi,Std, s - \frac{n-1}{2}) = \int_{\GL_n(F)}{\Phi(g) \langle v_0^\vee, \pi(g) v_0 \rangle |\det(g)|^{s}\,dg}.\end{equation}
Here $\Phi$ is the characteristic function of $M_n(\mathcal{O}_F)$ and $\langle \cdot, \cdot \rangle$ is the canonical pairing between $\pi$ and its contragredient.  The purpose of this note is to explain how (\ref{GJ1}) has an analogue for quasisplit groups $\GSpin(n)$, and to give a few easy applications of this analogue.

Continue to assume that $F$ is a $p$-adic local field.  Suppose $G$ is a reductive $F$-group, $\pi$ is an unramified irreducible representation of $G(F)$, and $v_0, v_0^\vee$ are spherical vectors in $\pi$ and its contragredient with $\langle v_0^\vee, v_0 \rangle =1$.  Suppose $\rho: \,^LG \rightarrow \GL_N(\C)$ is a representation of the Langlands dual group of $G$.  It follows from the Satake isomorphism that there is a unique function $\Delta_{\rho}^s: G(F) \rightarrow \C$ that is bi-invariant under a fixed maximal compact and that satisfies
\[L(\pi,\rho,s) = \int_{G(F)}{\Delta_{\rho}^s(g) \langle v_0^\vee, \pi(g) v_0 \rangle \,dg}.\]
One can give explicit expressions for the $\Delta_{\rho}^s$ \cite{Li2016,sakellaridis}, which however are typically complicated.  Since these expressions are complicated, they can be difficult to use in certain computations\footnote{For instance, it would seem to be difficult to use such expressions in Lemma \ref{FClemma} below, which is then used in section \ref{section:globInt}.}.

The doubling integrals of Piatetski-Shapiro and Rallis \cite{gpsrDoubling} give relatively simple expressions for $\Delta_{\rho}^s$ when $\rho$ is the standard representation of the dual group of a classical group.  For example, consider the case of split $\SO(2n)$, the group of $g \in \SL_{2n}$ that satisfy $g\mm{}{1_n}{1_n}{} g^{t} = \mm{}{1_n}{1_n}{}$.  For $g\in \SO(2n)$, define the function $\Delta_{s}(g)$ as follows.  By the Cartan decomposition, $g = k_1 t k_2$, with $k_1, k_2$ in the maximal compact $K$ and $t = \diag(t_1,t_2, \ldots t_n, t_1^{-1}, t_2^{-1}, \ldots, t_n^{-1})$, with the $t_i \in \mathcal{O}_F$.  Now set 
\begin{equation}\label{deltaSO2n} \Delta_s(g) = |t_1 \cdots t_n|^{s}, \text{ when } g = k_1 tk_2, t = \diag(t_1,t_2, \ldots t_n, t_1^{-1}, t_2^{-1}, \ldots, t_n^{-1}), t_i \in \mathcal{O}_F.\end{equation}  Here $| \cdot |$ is the normalized $p$-adic absolute value on $F$.  It is proved in \cite{gpsrDoubling} that 
\begin{equation}\label{deltaSO2nL}\frac{L(\pi,Std,s+1-n)}{A(s)} = \int_{\SO(V)(F)}{\Delta_s(g) \langle v_0^\vee, \pi(g) v_0\rangle \,dg}\end{equation}
where $A(s)$ is a certain explicit product of abelian $L$-functions that do not depend on the Satake parameters of $\pi$.

By contrast, the $\Delta_{\rho}^s$ can be quite non-trivial outside of the case of standard representations of classical groups.  For example, such is the case for the degree $7$ ``Standard" $L$-function on $G_2$ \cite{gurSeg}, or the degree $8$ ``Spin" $L$-function on $\GSp_6$ \cite{andrianov}, and this nontriviality leads to the principal difficulty that must be overcome in the papers \cite{gurSeg,pollackShah}.

Suppose now that $(V,q)$ is a quadratic space, i.e., that $V$ is a finite-dimensional $F$ vector space and $q$ is a non-degenerate quadratic form on $V$.  The purpose of this note is to explain how a simple observation regarding the function $\Delta_{s}$ from (\ref{deltaSO2n}) can make calculations involving the standard $L$-function on orthogonal groups easier.  First, recall the group $\GSpin(V)$, which is a central $\GL_1$-extension of $\SO(V)$,
\begin{equation}\label{exactSeq}1 \rightarrow \GL_1 \rightarrow \GSpin(V) \rightarrow \SO(V) \rightarrow 1.\end{equation}
The group $\GSpin(V)$ may be defined concretely inside the even Clifford algebra $\Clif^{+}(V,q)$ constructed from $V,q$; see section \ref{GSpinVdefs}.  It has a similitude character $\nu: \GSpin(V) \rightarrow \GL_1$, the kernel of which is the group $\Spin(V)$. The observation about which this paper centers is the following.  Fix a maximal integral $\mathcal{O}_F$ lattice $\Lambda$ in $V$, which is specified precisely in subsection \ref{subsec:Ldef} below.  Then we may form the Clifford algebra $\Clif(\Lambda,q)$, a finite $\mathcal{O}_F$-module, and one has $\Clif(V,q) = \Clif(\Lambda,q) \otimes_{\mathcal{O}_F} F$.  Define $\Phi$ to be the characteristic function of $\Clif(\Lambda,q)$, i.e., for $g$ in $\GSpin(V)$, define $\Phi(g)$ to be $1$ if $g \in \Clif(\Lambda,q)$ and $0$ otherwise. 
\begin{theorem}\label{GJGSpin} Suppose $F$ is a $p$-adic local field, $V$ is such that $\SO(V)$ is quasisplit, $\dim(V) \geq 2$, and $\pi$ is a spherical representation of $\GSpin(V)$.  Under these assumptions, $L(\pi,Std,s)$ is defined, and for $Re(s)$ sufficiently large, one has
\begin{equation}\label{GJSpinInt}\frac{L(\pi,Std,s+1-\dim V/2)}{d^V(s)} = \int_{\GSpin(V)(F)}{\Phi(g) |\nu(g)|^{s} \langle v_0^\vee, \pi(g) v_0\rangle \,dg}\end{equation}
where
\[d^V(s) = \prod_{2 \leq j \leq \dim V-2: j \equiv 0 (2)}{\zeta(\omega_{\pi},2s-j)}.\]
\end{theorem}
Here $\omega_{\pi}$ is the central character of $\pi$ and $\zeta(\omega_{\pi},s) = (1-\omega_{\pi}(\varpi)|\varpi|^{s})^{-1}$, where $\varpi$ is a uniformizer of $F$. Due to the presence of the abelian zeta functions $d^V(s)$, we say that $|\nu(g)|^{s}\Phi(g)$ gives an ``approximate" local $L$-function on $\GSpin(V)$, as opposed to an ``exact" local $L$-function, which would not have the $\zeta(\omega,s)$ factors in the denominator.

Once formulated, Theorem \ref{GJGSpin} essentially follows immediately from the above result (\ref{deltaSO2nL}) of Piatetski-Shapiro and Rallis, by integrating over the center $\GL_1 \subseteq \GSpin(V)$. (In the case of $\GSpin_5 = \GSp_4$, Theorem \ref{GJGSpin} goes back at least to Shimura \cite{shimura_MC} and Satake \cite[Appendix 1]{satake}.) Since the proof of (\ref{deltaSO2nL}) in \cite{gpsrDoubling} is a little involved, and our objective is to show the usefulness of considering the function $\Phi(g)|\nu(g)|^s$, we give a simple direct proof of Theorem \ref{GJGSpin} in section \ref{section:Lfcn}.  The similarity of (\ref{GJSpinInt}) with the case (\ref{GJ1}) of $\GL_n$ is clear: $\GL_n$ is replaced with $\GSpin(V)$, $\det: \GL_n \rightarrow \GL_1$ with $\nu: \GSpin(V) \rightarrow \GL_1$, and $M_n(F)$ replaced by $\Clif(V)$.  We remark that lifting from $\SO(V)$ to $\GSpin(V)$, which has the similitude character $\nu: \GSpin(V) \rightarrow \GL_1$, comports with the philosophy of \cite{bk2, ngo}.  We should point out, however, that one important difference with this setup compared with the case of $\GL_n$ is that $\GSpin(V)$ is not Zariski open in $\Clif(V)$ or $\Clif^{+}(V)$.

Since the function $d^V(s)$ of Theorem \ref{GJGSpin} depends only upon the central character of $\pi$, it is an easy matter to replace $\Phi(g)$ by a function $\Phi'(g)$ so that $|\nu(g)|^{s}\Phi'(g)$ gives an exact local $L$-function.  Since these ``basic functions" $\Phi'(g)$ have become of interest \cite{casselman, Li2016,sakellaridis}, we mention now the explicit formula for $\Phi'$. Choose a uniformizer $\varpi$ for $F$, set $q = |\varpi|^{-1}$, and denote by $\Phi_{n}$ the characteristic function of $\varpi^n \Clif(\Lambda) \setminus \varpi^{n+1}\Clif(\Lambda)$, so that $\Phi = \sum_{n \geq 0}{\Phi_n}$.  Define polynomials
$p'_{V,M}(t)$ by the generating series
\[\sum_{M \geq 0}{p'_{M,V}(t)X^M} = \prod_{0 \leq j \leq \dim(V)-2: j \equiv 0 (2)}{(1-t^j X)^{-1}},\]
and set $\Phi'(g) = \sum_{M \geq 0}{p'_{V,M}(q)\Phi_{M}(g)}$. Then one deduces from Theorem \ref{GJGSpin} that
\begin{equation}\label{phi':intro}L(\pi,Std,s+1-\dim V/2) = \int_{\GSpin(V)(F)}{\Phi'(g) |\nu(g)|^{s} \langle v_0^\vee, \pi(g) v_0\rangle \,dg}.\end{equation}
See Remark \ref{rmk:exactL}.  Since we only use the ``approximate" $L$-function $\Phi(g)|\nu(g)|^{s}$ below, and not the ``exact" $L$-function $\Phi'(g)|\nu(g)|^{s}$, we have not attempted to deduce the expression (\ref{phi':intro}) from \cite{bk3}, \cite{sakellaridis}, or \cite{Li2016}.

The body of this paper consists of showing how one can apply Theorem \ref{GJGSpin}, together with the pullback formula \cite{garrett2, shimuraBook} (reviewed in section \ref{section:pullback}), to prove somewhat efficiently that various Rankin-Selberg integrals on $\SO(V)$ represent its standard $L$-function.  We give three old examples, and one new example.  The three old examples are \ref{item:double}, \ref{item:period}, and \ref{item:bessel} of the following list, while the new example is \ref{item:KS}.
\begin{enumerate}[label=(\alph*)]
\item \label{item:double} The doubling integral \cite{gpsrDoubling}, in which one puts cusp forms on $\SO(n) \times \SO(n)$ and a Siegel Eisenstein series on $\SO(2n)$;
\item \label{item:period} The integral of Murase-Sugano \cite{muraseSugano}, which is vastly generalized in \cite{gpsr}, in which one puts a cusp form on $\SO(n)$ and a certain Eisenstein series on $\SO(n+1)$;
\item \label{item:bessel} The integral of Sugano \cite{sugano}, in which one puts a cusp form on $\SO(n)$ and a certain Eisenstein series on $\SO(n-1)$. 
\item \label{item:KS} A generalization of the integral of Kohnen-Skoruppa \cite{ks} (see also \cite{pollackShahKS}) from $\GSp_4 = \GSpin_5$ to $\GSpin(V)$, $\dim V \geq 5$.  This integral involves a cusp form, an Eisenstein series, and a special function all on $\GSpin(V)$.
\end{enumerate}
For \ref{item:double} (resp. \ref{item:period}), one could instead proceed as efficiently by using the result (\ref{deltaSO2nL}) of Piatetski-Shapiro and Rallis (resp. the result (\ref{deltaSO2nL}) and the pullback formula).  In examples \ref{item:bessel} and \ref{item:KS}, one seems to obtain a genuine advantage by using Theorem \ref{GJGSpin}.

The above-mentioned computations may be contrasted with those of Braverman-Kazhdan \cite[Section 7]{bk3}.  The authors of \emph{loc. cit.} give an exact expression for $\Delta_{Std}^s$ on classical groups (including $\GSpin(V)$), which is well-suited to computing with the doubling integral of Piatetski-Shapiro--Rallis.  Furthermore, they use this expression to reprove the analytic properties of the global $L$-function $L(\pi,Std,s)$ via the method of Godement-Jacquet.  The main point of this note is that by using the approximate Godement-Jacquet function $\Phi(g)|\nu(g)|^s$ instead of the exact function of \cite{bk3}, one can more easily compute in a wider variety of scenarios (albeit with weaker results.)

We emphasize that not much in this paper consists of proofs of new theorems; the standard $L$-functions of automorphic forms on orthogonal groups are very well understood, for example from \cite{gpsr}.  Our goal is simply to impress upon the reader the similarity between $\GL_n$ and $\GSpin(V)$ in the context of Godement-Jacquet theory, and to show the usefulness of using Schwartz-Bruhat functions on $\Clif(V)$ to compute with the standard $L$-function on $\SO(V)$.  It is an interesting question to see if some other aspects of Godement-Jacquet theory on $\GL_n$ can be extended to $\SO(V)$ using Schwartz-Bruhat functions on $\Clif(V)$.  The reader should see \cite{getz} for some global aspects, and also the paper \cite{li_FE} for some different local aspects.

We now give a summary of the layout of the paper.  In section \ref{section:prelim}, we give a few preliminaries on the group $\GSpin(V)$, including its definition and the definition of the standard $L$-function.  In section \ref{section:twoLemmas}, we give two lemmas that are used in rest of the paper.  In section \ref{section:Lfcn} we prove Theorem \ref{GJGSpin}, and in section \ref{section:pullback} we review the pullback formula.  Finally, in section \ref{section:globInt}, we explain how the pullback formula and Theorem \ref{GJGSpin} can be used to quickly calculate the four global integrals mentioned above.

$\textrm{ }$

\noindent \textbf{Acknowledgments} We thank Nadya Gurevich for enlightening conversations at MSRI in 2014.  We also thank William Casselman, James Cogdell, Paul Garrett, Jayce Getz, Wen-Wei Li, and Shrenik Shah for their helpful comments on an earlier version of this manuscript.  Finally, we thank the anonymous referee for his or her comments, which have improved the writing of this paper.

\section{Preliminaries}\label{section:prelim}
In this section we give some preliminaries on the group $\GSpin(V)$ and its standard $L$-function.  For a primer on Clifford algebras, the reader might see \cite[Chapters 1-3]{meinrenken}.

\subsection{The group $\GSpin(V)$}\label{GSpinVdefs}
Suppose $V,q$ is a quadratic space. For $v_1, v_2 \in V$ we set $(v_1, v_2) = q(v_1+v_2) - q(v_1) - q(v_2)$ the induced symmetric bilinear form on $V$.

The group $\GSpin(V)$ may be defined concretely as follows.  Set $\Clif(V,q)$ the Clifford algebra of $(V,q)$, the quotient of the tensor algebra $T(V)$ of $V$ by the two-sided ideal generated by $v^2 -q(v)$, for $v \in V$.  We frequently write $\Clif(V)$ for $\Clif(V,q)$, because the quadratic form $q$ will be fixed.  The map $V \rightarrow T(V)$ induces a canonical injection $V \rightarrow \Clif(V)$.  Denote by $\Clif^{\pm}(V,q)$ the even and odd degree parts, so that $\Clif(V) = \Clif^{+}(V) \oplus \Clif^{-}(V)$. The Clifford algebra $\Clif(V,q)$ has an involution, $y \mapsto y^*$, induced by reversing the order of products in the tensor algebra $T(V)$: If $v_j \in V$, $(v_1 \cdots v_r)^* = v_r \cdots v_1$.  The algebraic group $\GSpin(V)$ is defined as
\[\GSpin(V) = \{(g,\nu) \in \Clif^{+}(V) \times \GL_1: g^*g = \nu \text{ and } g^*Vg \subseteq V\}.\]
Note that $g^{-1} = \nu^{-1}(g) g^*$.  The character $\nu: \GSpin(V) \rightarrow \GL_1$ is called the similitude, and $\Spin(V)$ is the subgroup of $\GSpin(V)$ where $\nu =1$.  Note that if $V' \rightarrow V$ is an isometric inclusion of non-degenerate quadratic spaces, this inclusion induces maps $\Clif(V') \rightarrow \Clif(V)$, $\Clif^{\pm}(V') \rightarrow \Clif^{\pm}(V)$, and $\GSpin(V') \rightarrow \GSpin(V)$.

The conjugation action of $\GSpin(V)$ on the embedded $V \subseteq \Clif^{-}(V)$, $v \cdot g = g^{-1}vg$, defines a map $\GSpin(V) \rightarrow \GL(V)$ whose image is immediately seen to land in the orthogonal group $\mathrm{O}(V)$.  Set $Z_{V} \simeq \GL_1 \subseteq \GSpin(V)$ to be the group of $(t,t^2) \in \Clif^{+}(V) \times \GL_1$, which is clearly in the center of $\GSpin(V)$.  This conjugation action of $\GSpin(V)$ on $V$ gives rise to the sequence 
\begin{equation}\label{SEStoSOV}1 \rightarrow Z_{V} \rightarrow \GSpin(V) \rightarrow \SO(V) \rightarrow 1,\end{equation}
exhibiting $\GSpin(V)$ as a central $\GL_1$-extension of $\SO(V)$. Note that we always let $\GSpin(V)$ and $\SO(V)$ act on the right of $V$.

\subsection{Parabolic subgroups}\label{subsec:parabolic}  Suppose $\mathcal{U}: U_1 \subseteq U_2 \subseteq \cdots \subseteq U_r$ is a chain of totally isotropic subspaces of $V$.  Define $P_{\mathcal{U}}$ to be the subgroup of $\SO(V)$ or $\GSpin(V)$ (depending on context) that stabilizes this isotropic flag.  Then $P_{\mathcal{U}}$ is a parabolic subgroup, and every parabolic subgroup arises in this way.  See \cite{conrad}, for example, for a clear treatment of the parabolic subgroups of the special orthogonal group.

Suppose $U$ is an isotropic subspace; we consider the parabolic $P_U$.  The unipotent radical $N_{U}$ of $P_U$ has a two-step filtration $N_U \supseteq N_U' \supseteq 1$.  The exponential map $\exp: U^\perp \otimes U \rightarrow N_U$ considered inside the Clifford algebra identifies $\wedge^2 U \simeq N_U'$ and $(U^\perp/U) \otimes U$ with $N_U/N_U'$. If $U$ is $1$-dimensional, then $N_U' = 1$, while if $\dim(V)$ is even and $\dim(U) = \dim(V)/2$, $N_U = N_U'$.  In all other cases, both $N_U/N_U'$ and $N_U'$ are nontrivial.

We now make some special notation for the parabolic $P = P_{Ff_1}$ that stabilizes the isotropic line $Ff_1$.  Choose an isotropic $e_1$ with $(e_1,f_1) = 1$, and denote by $V_1 \subseteq V$ the perpendicular space to $Fe_1 \oplus Ff_1$, so that $V = Fe_1 \oplus V_1 \oplus Ff_1$.  Write $N$ for the unipotent radical of $P$ and $M$ for the Levi subgroup of $P$ that stabilizes the decomposition $V = Fe_1 \oplus V_1 \oplus Ff_1$.  The map $n: V_1 \rightarrow N$, $x \mapsto n(x) = 1 + f_1x$ identifies $V_1$ with $N$.  The map $m: \GL_1 \times \GSpin(V_1) \rightarrow M$ defined by $m(\lambda,y) = m_1^*(\lambda)y = (e_1f_1 + \lambda f_1 e_1) y$ identifies the Levi subgroup $M$ with $\GL_1 \times \GSpin(V_1)$.  Here and below, multiplications such as $f_1x$ and $e_1f_1$ take place inside the $\Clif(V)$.

Represent elements of $V$ as triples $(\alpha,v,\delta)$ according to the decomposition $V = Fe_1 \oplus V_1 \oplus Ff_1$.  We compute
\begin{align*} (\alpha,v,\delta) \cdot n(x) &= (\alpha e_1 + v + \delta f_1) \cdot n(x) \\ &= (\alpha, v+ \alpha x, \delta - (v,x) - \alpha q(x)).\end{align*}
One has $\nu(m(\lambda,y)) = \lambda \nu(y)$, and
\[ e_1 \cdot m(\lambda, y) = \lambda e_1; \quad v \cdot m(\lambda, y) = v \cdot y; \quad f_1 \cdot m(\lambda, y) = \lambda^{-1} f_1.\]

Suppose $\dim V = 2n+1$ is odd.  Then $\SO(V)$ (and thus $\GSpin(V)$) is quasisplit precisely when it is split, which occurs when $V$ has an isotropic subspace of dimension $n$.  When $\dim V = 2n+2$ is even, then $\SO(V)$ and $\GSpin(V)$ are quasisplit when $V$ has an $n$-dimensional isotropic subspace, say $U$.  Then $V = U^\vee \oplus V_E \oplus U$, where 
\begin{itemize}
\item $U^\vee$ is also isotropic, and is canonically isomorphic to $\mathrm{Hom}_{F}(U,F)$ under the symmetric pairing;
\item $V_E$ is orthogonal to $U^\vee \oplus U$;
\item there is a (unique) quadratic \'etale $F$-algebra $E$ so that $V_E$ with its induced quadratic form $q|_{V_E}$ is isometric with $E$ and its norm form: $n_{E/F}: E \rightarrow F$.\end{itemize}
The group $\GSpin(V)$ is split precisely when $E = F \times F$ is the split quadratic \'{e}tale $F$-algebra.  Below, we abuse language and say that $V$ is split or quasisplit if $\SO(V)$ (and thus $\GSpin(V)$) is so.

\subsection{Maximal torus}\label{subsec:maxtorus} Suppose $V = U^\vee \oplus V_{E} \oplus U$ is quasisplit, so that $E$ is a quadratic \'{e}tale extension of $F$ if $\dim(V) = 2n+2$ is even, and is $F$ if $\dim(V)=2n+1$ is odd.  Write $N_{E/F}: E \rightarrow F$ for the norm map when $E$ is quadratic over $F$ and for the map $t \mapsto t^2$ when $E=F$.  Let $e_1,\ldots, e_n$ be a basis of $U^\vee$ and $f_1, \ldots f_n$ the dual basis of $U$ so that $(e_i,f_j) = \delta_{ij}$.  Denote by $T$ the subgroup of $\GSpin(V)$ that preserves the decomposition $V = Fe_1 \oplus \cdots \oplus Fe_n \oplus V_{E} \oplus Ff_n \oplus \cdots \oplus Ff_1$.  Then $T$ is a maximal torus of $\GSpin(V)$, and is a maximal split torus of $\GSpin(V)$ when $V$ is split.

For $i=1, 2, \ldots, n$, define $m_i: \GL_1 \rightarrow \GSpin(V)$ via $m_i(t) = te_i f_i + f_i e_i$ and $m_i^*(t) = (m_i(t))^* = e_if_i+tf_ie_i$.  Furthermore, define $z: \GL_1 \rightarrow Z_{V} \subseteq \GSpin(V)$ via $z(t) = t$.  One has $m_i m_i^* = z$, and $m_1, \ldots, m_n$ and $z$ land in $T$. If $\dim(V) = 2n+1$ is odd, the cocharacters $z, m_1, \ldots, m_n$ form a basis of the cocharacters of $T$.

If $E$ is as above, $\GSpin(E)$ is canonically isomorphic with $E^\times$, as the following lemma shows.
\begin{lemma}\label{iota0} Suppose $E=F$ or $E$ is a quadratic \'etale extension $F$, and write $V_{E}$ for the quadratic space $E$ with quadratic form given by the norm form of $E$.  Write $\iota_1: V_{E} \rightarrow \Clif(V_{E})$ for the standard inclusion.  Then the map $\iota_0: E \rightarrow \Clif^{+}(V_{E})$ defined by $\iota_0(\lambda) = \iota_1(1)\iota_1(\lambda)$ is an algebra isomorphism, with $\iota_0(\lambda)\iota_0(\lambda)^*= N_{E/F}(\lambda)$.  It follows that $\iota_0$ restricts to an isomorphism $E^\times \simeq \GSpin(V_{E})$.  Furthermore, $\iota_0$ restricts to an isomorphism $\mathcal{O}_E \rightarrow \Clif^+(\mathcal{O}_E)$. \end{lemma}
\begin{proof} It is clear that $\iota_0$ is a linear isomorphism, and that is has the stated property about norms.  Since $E$ is a composition algebra, the ring isomorphism follows.  The integrality statement is also clear.

Since $V_{E} = \iota_1(V_E) = \Clif^{-}(V_E)$ in both cases $\dim E = 1$ or $\dim E = 2$, the condition $g^*V_{E} g \subseteq V_{E}$ on $\GSpin(V_E)$ is automatically satisfied.  The isomorphism $\iota_0: E^\times \rightarrow \GSpin(V_{E})$ follows.
\end{proof}

Let $z_{E}: E^\times \rightarrow \GSpin(V_E) \rightarrow T \subseteq \GSpin(V)$ denote the cocharacter of $T$ induced by the map $\iota_0$ of Lemma \ref{iota0}.  In both the case of $\dim(V)$ even or odd, one has a direct product $T = m_1 \times m_2 \times \cdots m_n \times z_{E}$.  (When $\dim(V)$ is odd, $z_E = z_{V}$ is the central cocharacter.)

\subsection{Maximal compact and $L$-function}\label{subsec:Ldef} Suppose $F$ is a $p$-adic local field.  Let $e_i, f_j$ and $E$ be as above.  Define the lattice $\Lambda \subseteq V$ as 
\[\Lambda = \mathcal{O}_{F}e_1 \oplus \cdots \oplus \mathcal{O}_{F}e_n \oplus V_{\mathcal{O}_E} \oplus \mathcal{O}_{F}f_n \oplus \cdots \oplus \mathcal{O}_{F}f_1,\]
where $V_{\mathcal{O}_E} \subseteq V_{E}$ is the submodule identified with $\mathcal{O}_E$ under the fixed identification of $V_E$ with $E$.  Set $K$ to be the subgroup of $\GSpin(V)$ that stabilizes the $\mathcal{O}_F$-lattice $\Lambda$, and for which the similitude lands in $\mathcal{O}_F^\times$.  Then $K$ is a maximal compact subgroup of $\GSpin(V)(F)$ \cite{hinaSugano}, and one has an Iwasawa decomposition $\GSpin(V)(F) = B(F)K$, where $B$ is the Borel subgroup of $\GSpin(V)$ that stabilizes the isotropic flag $U_1 \subseteq \cdots \subseteq U_n$, where $U_j = F f_1 \oplus \cdots \oplus Ff_j$.

Suppose that $\pi$ is an irreducible $K$-spherical representation of $\GSpin(V)(F)$.  We recall the standard $L$-function of $\pi$.  Such a $\pi$ is the spherical subresentation of $Ind_{B}^{\GSpin(V)}(\delta_{B}^{1/2}\alpha)$ for an unramified character $\alpha: T(F)\rightarrow \C^\times$.  If $\pi$ is such, define
\[L(\pi,Std,s) = L_{E,*}(\alpha_E,s)\prod_{1 \leq i \leq n}{L(\alpha_i,s)L(\alpha_i',s)}\]
where
\begin{itemize}
\item $\alpha_E: E^\times \rightarrow \C^\times$, $\alpha_i: F^\times \rightarrow \C^\times$, $\alpha_i': F^\times \rightarrow \C^\times$ are the pullbacks of $\alpha$ via the maps $z_E: E^\times \rightarrow T(F)$ $m_i: F^\times \rightarrow T(F)$, and $m_i^*: F^\times \rightarrow T(F)$;
\item $L(\alpha_i,s)$ and $L(\alpha_i',s)$ are the usual Tate local $L$-functions.  That is, $L(\alpha_i,s) = (1-\alpha_i(\varpi)|\varpi|^s)^{-1}$ for a uniformizer $\varpi$ of $F$ and similarly for $L(\alpha_i',s)$.
\item $L_{E,*}(\alpha_{E},s)$ is defined to be the usual Tate local $L$-function $L_{E}(\alpha_{E},s)$ of $\alpha_{E}$ if $E/F$ is quadratic, and is defined to be $1$ if $E=F$.  This is, if $E$ is a quadratic field extension of $F$, $L_{E,*}(\alpha_E,s) =(1-\alpha_{E}(\varpi)|\varpi|^{2s})^{-1}$, and if $E = F\times F$ with uniformizers $\varpi_1, \varpi_2$, $L_{E,*}(\alpha_{E},s) = (1-\alpha_{E}(\varpi_1)|\varpi|^s)^{-1}(1-\alpha_{E}(\varpi_2)|\varpi|^s)^{-1}$.\end{itemize}

\subsection{A filtration on $\Clif(V)$}\label{subsec:filtration} The results in this subsection will only be used in section \ref{section:pullback}.  Suppose $U \subseteq V$ is an isotropic subspace.  Associated to $U$, one can define an increasing filtration $W^{U}_{\bullet}$ on $\Clif(V)$ as follows. First define an increasing filtration on $V$ via $W^U_{-1}V = U$, $W^U_0 V = U^\perp$, and $W^U_{1} V = V$.  Define $W_k^{U}V = 0$ if $k < -1$, and $W_k^U V = V$ if $k > 1$.  The filtration $W^U$ on $V$ induces a filtration on $\Clif(V)$ via
\begin{equation}\label{FilUdef} W^U_k \Clif(V) = \{x \in \Clif(V): x = \sum{x_1 x_2 \cdots x_r}, x_i \in W^U_{\alpha_i}V, \alpha_1 + \cdots + \alpha_r \leq k\}.\end{equation}
Note that the filtered pieces $W^U_k \Clif(V)$ are stable under the involution $*$.  We will see momentarily that the filtration $W^U$ on $\Clif(V)$ restricts to the filtration $W^U$ on $V$. 

First, choose $V_0, U^\vee$ so that $V = U^\vee \oplus V_0 \oplus U$ and $U^\perp = V_0 \oplus U$.  Denote by $h_U: \GL_1 \rightarrow \GSpin(V)$ the map satisfying $\nu(h_U(t)) = t^{\dim(U)}$ and
\[h_U(t)^{-1}v h_U(t) = \begin{cases} tv &\mbox{if } v \in U^\vee\\ v &\mbox{if } v \in V_0\\ t^{-1} v &\mbox{if } v \in U.\end{cases}\] 
(One can explicitly construct such an $h_U$ using the maps $m_i(t)$ of subsection \ref{subsec:maxtorus}, and thus $h_U$ exists.  It is unique by (\ref{SEStoSOV}).) Now, the fact that the filtration $W^U$ on $\Clif(V)$ restricts to the filtration $W^U$ on $V$ follows by considering the conjugation action of this cocharacter on the expressions in (\ref{FilUdef}).

\begin{definition} For $U$, $W^U$ as above, define $P(U,V) = W^U_0\Clif(V)$ and $N(U,V) = W^U_{-1}\Clif(V)$.\end{definition}
One has the following lemma.  Recall that $P_U$ denotes the parabolic subgroup of $\GSpin(V)$ that stabilizes the isotropic subspace $U$ of $V$ and $N_U$ denotes its unipotent radical.
\begin{lemma}\label{PUVlemma} The space $P(U,V)$ is an algebra, and $N(U,V)$ is a two-sided ideal of $P(U,V)$.  Furthermore, the parabolic subgroup $P_U = P(U,V) \cap \GSpin(V)$, and its unipotent radical $N_U = (1 + N(U,V)) \cap \GSpin(V)$.\end{lemma}
\begin{proof} Since $W_j\Clif(V) \cdot  W_k\Clif(V) \subseteq W_{j+k}\Clif(V)$, it is clear that $P(U,V)$ is an algebra and $N(U,V)$ a two-sided ideal of $P(U,V)$.

For the second part, suppose $g \in P(U,V) \cap \GSpin(V)$, and $u \in U$.  Then $g^*ug =v$ for some $v \in V$.  But the left-hand-side of this equality is in $W_{-1}$, and thus $v \in V \cap W_{-1}\Clif(V) = U$.  Hence $g \in P_U$.

Conversely, suppose $g \in P_U = M_UN_U$, where $M_U$ is the Levi subgroup stabilizing the decomposition $V = U^\vee \oplus V_0 \oplus U$. Then $g = mn$ with $n \in \exp(U^\perp \cdot U) \subseteq W_{0}\Clif(V)$ and $m \in M_U$.  But, by the definition of $h_U$, $m$ commutes with $h_U(t)$.  By using the linear isomorphism $\Clif(V) = \bigwedge^{\bullet}(U^\vee) \otimes \Clif(V_0) \otimes \bigwedge^{\bullet}(U)$, one sees that this implies $m \in W_{0}\Clif(V)$ as well.

The inclusion $N_U \subseteq (1+N(U,V)) \cap \GSpin(V)$ follows from the fact that $N_U = \exp(U^\perp \cdot U)$.  Conversely, suppose $n = 1 + x$ is in $\GSpin(V)$, with $x \in N(U,V)$.  Then $\nu(n) = 1$, and if $v \in V$, $v \cdot n = (1+x^*)v(1+x) = v + x^*v + vx + x^*x$.  Hence $v \cdot n$ has the same image in the associated graded of $W^U_{\bullet}$ on $V$ that $v$ does, so $n \in N_U$.\end{proof}

\section{Two lemmas}\label{section:twoLemmas} In this section, we give two lemmas which will be utilized in other sections. We state and prove these lemmas now so that section \ref{section:globInt} and section \ref{section:Lfcn} may be read briskly.  However, the reader may wish to proceed directly to other sections and then come back to this section when needed.

We require the following definition.
\begin{definition} Suppose $F$ is a $p$-adic local field, the lattice $\Lambda \subseteq V$ is fixed, and $y \in \GSpin(V)$.  Then, there is a unique integer $k$ so that $y = \varpi^k y_0$ with $y_0 \in \Clif(\Lambda)$ and $y_0$ has nonzero reduction in $\Clif(\Lambda/\varpi\Lambda)$.  We set $\val(y) = k$ and $||y|| = |\varpi|^{\val(y)} = |\varpi|^{k}$.\end{definition}
We set $q =\# \mathcal{O}_F/\varpi = |\varpi|^{-1}$.

\subsection{A volume calculation} Let $V$ be as above, and $V_1$ the span of $e_i, f_j$ with $j \geq 2$ and $V_{E}$, so that $V = Fe_1 \oplus V_1 \oplus Ff_1$.  Set $\Lambda_1 = \Lambda \cap V_1$, and for $y \in \GSpin(V_1)$ define
\[U_{y} = \{x \in V_1: xy \in \Clif(\Lambda_1)\}.\]

\begin{lemma}\label{meas} Suppose $F$ is a $p$-adic local field and $V$ is quasisplit.  Let $y$ in $\GSpin(V_1)$ be equal to $\varpi^k y_0$, with $\val(y_0) = 0$.  Normalize an additive Haar measure on $V_1$ by setting $\Lambda_1$ to have measure $1$.  Then the measure of $U_y$ is 
\[q^{k \dim V_1} |\nu(y_0)|^{-1} = (q^k)^{\dim V_1 -2} |\nu(y)|^{-1} = ||y||^{2-\dim V_1}|\nu(y)|^{-1}.\]
\end{lemma}
\begin{proof} First, one reduces immediately to the case of $k = 0$.  Now, assume $k = 0$ for the rest of the proof, so that $y = y_0$ has $\val(y) = 0$.  Let $K_1 \subseteq \GSpin(V_1)$ be the maximal compact defined by $\Lambda_1$.  If $k \in K_1$, then $U_{yk} = U_{y}$.  Moreover, the map $x \mapsto x \cdot k = k^{-1}xk$ defines a bijection $U_{ky} \rightarrow U_{y}$.  Since this linear map on $V_1$ is measure-preserving, we deduce that the function $y \mapsto \mathrm{meas}(U_y)$ is a $K_1$ bi-invariant function on $\GSpin(V_1)$.  Thus, we may apply the Cartan decomposition to simplify the computation.

Now, we may write $\Lambda_1$ as an orthogonal direct sum of lattices $\Lambda^i = \mathcal{O}_F e_i \oplus \mathcal{O}_F f_i$ and $V_{\mathcal{O}_E}$, which is free of rank $1$ or $2$ over $\mathcal{O}_F$.  Applying the Cartan decomposition, we may assume $y = \prod_{i}{y_i}$, with $\val(y_i) = 0$ and $y_i \in \GSpin(\Lambda^i \otimes F) \subseteq \GSpin(V_1)$.  Since the $\val(y_i)$ are all $0$, one checks easily that $U_y = U_{y_1} \times \cdots \times U_{y_j}$.  Hence we have reduced to the case where $\dim V_1 = 1$ or $2$.

To handle the case $\dim(V_1) =1$ or $2$, we may assume $V_1 = V_{E}$ for $E=F$ or a quadratic \'{e}tale extension of $F$.  In this case, $\Clif(V_E)^{-} = V_E$.  Thus here we have $U_y = \{x \in Clif(V_E)^{-}: xy \in \Clif(\Lambda)^{-}\}$.  Hence the measure of $U_{y}$ is $|\det(y: \Clif(V_E)^{-})|^{-1}$, the inverse of the absolute value of the determinant of $y$ acting by multiplication on $\Clif(V_E)^{-}$.  But via the map $\iota_0$ of Lemma \ref{iota0}, this map is the right multiplication of $E^\times$ on $V_E\simeq E$, so its determinant is $n_{E/F}(\iota_0^{-1}(y)) = \nu(y)$, as desired.\end{proof}

\subsection{A Fourier coefficient calculation} In this subsection, $F$ remains a $p$-adic local field, and $V$ quasisplit. Recall that $P = MN$ is the parabolic subgroup that stabilizes the line $Ff_1$.  Recall the function $\Phi$, which is the characteristic function of $\Clif(\Lambda) \subseteq \Clif(V)$, and suppose $T \in V_1$ has $q(T) \neq 0$.  Define
\[S_{T}(\Phi)(m) = \int_{V_1}{\psi((T,x))\Phi(n(x)m)\,dx},\]
the $T^{th}$ Fourier coefficient of $\Phi$.  Here $\psi: F \rightarrow \C^\times$ is an additive character with conductor $\mathcal{O}_F$.  In this subsection we compute $S_{T}(\Phi)(m(\lambda,y))$.
\begin{claim}\label{lemClaim} One has $\Phi(n(x)m(\lambda,y)) = \charf(y,\lambda y)\charf(xy)$.   That is, $n(x)m(\lambda, y)$ is in $\Clif(\Lambda)$ if and only if $y, \lambda y$ and $xy$ are in $\Clif(\Lambda_1)$. \end{claim}
\begin{proof} First note that since $f_1 e_1 f_1 = f_1$ and $f_1^2 = 0$, $f_1x m(\lambda,y) = f_1x y = -xyf_1$. Hence
\begin{align*} n(x)m(\lambda,y) &= (1-xf_1)m(\lambda,y) \\ &= m(\lambda,y) - xf_1m(\lambda,y) \\ &= m(\lambda,y) - xyf_1 \\ &= (e_1f_1 + \lambda f_1e_1)y - xyf_1.\end{align*}
Thus if $y, \lambda y,$ and $xy$ are in $\Clif(\Lambda_1)$, $n(x)m(\lambda,y)$ is in $\Clif(\Lambda)$.

Conversely, suppose $n(x)m(\lambda,y) = (e_1f_1)y + (f_1e_1)(\lambda y) - (xy)f_1$ is in $\Clif(\Lambda)$.  Then multiplying on the right by $f_1$ proves that $f_1 (\lambda y) \in \Clif(\Lambda)$, and multiplying on the left by $f_1$ proves that $f_1 y$ is in $\Clif(\Lambda)$.  Multiplying by $e_1$ and subtracting proves that $(xy)f_1$ is in $\Clif(\Lambda)$.  Finally, the identity $r = (rf_1)e_1 \pm e_1(rf)$ for $r \in \Clif^{\pm}(V_1)$ gives that $y, \lambda y$ and $xy$ are in $\Clif(\Lambda_1)$, finishing the proof of the claim.\end{proof}

From Claim \ref{lemClaim} we obtain
\[S_T(\Phi)(m(\lambda,y)) = \int_{V_1(F)}{\psi((T,x))\Phi(n(x)m(\lambda,y))\,dx} =\charf(y,\lambda y) \int_{V_1(F)}{\psi((T,x))\charf(x y)\,dx}.\]
Thus the set of $x$ for which $\Phi(n(x)m) \neq 0$ is an abelian group.

Define
\[\Lambda_1^\vee = \{ v \in V_1(F): (v,x) \in \mathcal{O}_{F} \text{ for all } x \in \Lambda_1\}.\]
Note that for $S_{T}(\Phi)$ to be nonzero, one needs $T \in \Lambda_1^\vee$.  Assume $T \notin p \Lambda_1^\vee$.  Then one sees immediately that for $S_{T}(\Phi)(m(\lambda,y))$ to be nonzero, we need $\val(y) = 0$.  Hence, since $\lambda y$ must also be integral for $S_T(\Phi)(m(\lambda,y))$ to be nonzero, we deduce $\lambda \in \mathcal{O}_F$ as well.

We claim that $T \cdot y = y^{-1}Ty$ must be in $\Lambda_1^\vee$ for $S_T(\Phi)(m(\lambda,y)) \neq 0$.  To see this, make a variable change $x \mapsto yxy^{-1}$.  Then, $S_T(\Phi)$ becomes
\begin{align*} S_T(\Phi)(m(\lambda,y)) &= \int_{V_1(F)}{\psi((T,yxy^{-1}))\charf(yx)\,dx} \\ &= \int_{V_1(F)}{\psi((y^{-1}Ty,x))\charf(yx)\,dx}. \end{align*}
Hence we deduce $y^{-1}Ty \in \Lambda_1^\vee$ as well, for the Fourier coefficient $S_T(\Phi)$ to be nonzero.  

Summarizing the above argument, we have proved the first part of the following lemma.
\begin{lemma}\label{FClemma} Assume $T \in \Lambda_1^\vee$ but not in $p\Lambda_1^\vee$.  For $S_{T}(\Phi)(m(\lambda,y))$ to be nonzero, one needs $\val(y) = 0$, $\lambda \in \mathcal{O}_F$, and $y^{-1}Ty = T \cdot y$ in $\Lambda_1^\vee$. Furthermore assume that $\Lambda_1$ is self-dual, i.e., that $\Lambda_1^\vee = \Lambda_1$.  Then if these conditions are satisfied, $S_T(\Phi)(m(\lambda,y)) = |\nu(y)|^{-1}$.\end{lemma}
\begin{proof} We still must prove the second part of the lemma.  Suppose then that $\val(y) = 0$, $T \in \Lambda_1^\vee \setminus p\Lambda_1^\vee$, $\lambda \in \mathcal{O}_F$, and $y^{-1}Ty \in \Lambda_1^\vee$.  Then we claim that $xy \in \Clif(\Lambda_1)$ integral implies $(T,x)$ in $\mathcal{O}_F$.  Indeed, this follows from the identity
\[(T,x)y = (Tx + xT)y = T(xy) + (xy)(y^{-1}Ty) \in \Lambda_1^\vee \Clif(\Lambda_1) + \Clif(\Lambda_1)\Lambda_1^\vee\]
since then $(T,x)y \in \Clif(\Lambda_1)$, which implies $(T,x) \in \mathcal{O}_F$ since $\val(y) = 0$.

Thus $S_T(\Phi)(m(\lambda,y))$ is the measure of the set $\{x \in V_1(F): xy \in \Clif(\Lambda_1)\}.$  But this measure was computed in Lemma \ref{meas}, completing the proof. \end{proof}

\section{Proof of Theorem \ref{GJGSpin}}\label{section:Lfcn}
The purpose of this section is to give a direct proof of Theorem \ref{GJGSpin}.  Recall that we assume that $V$ is quasisplit, $\dim(V) = 2n+1$ or $\dim(V) = 2n+2$, and we keep notation as in section \ref{section:prelim}.  Recall that $B$ is a Borel subgroup, $T \subseteq B$ is a maximal torus, $\alpha: T(F) \rightarrow \C^\times$ is an unramified character, and $\pi = Ind_B^{\GSpin(V)}(\delta_{B}^{1/2}\alpha)$ is $K$-spherical.  We denote by $\phi_{\alpha}$ the spherical vector in this induction, so that $\phi_\alpha: \GSpin(V) \rightarrow \C$ satisfies $\phi_\alpha(bk) = (\delta_B^{1/2}\alpha)(b)$.  To prove Theorem \ref{GJGSpin} we must compute the integral
\[I(\alpha,s) = \int_{\GSpin(V)(F)}{\phi_\alpha(g) \Phi(g)|\nu(g)|^{s}\,dg}.\]
Finally, recall that $P = MN$ denotes the parabolic subgroup of $\GSpin(V)$ that stabilizes the line $Ff_1$.  We write $\tau = Ind_{B}^{M}(\delta_{P}^{-1/2}\delta_{B}^{1/2}\alpha)$ for the $K_1$-spherical representation of $M$ so that $Ind_P^{\GSpin(V)}(\delta_{P}^{1/2}\tau) = \pi$.

The proof of Theorem \ref{GJGSpin} is by induction on the rank of $\GSpin(V)$, by comparing $\GSpin(V)$ with $\GSpin(V_1)$.  We first dispose of the cases when $n=0$, so that $V = V_{E}$ with $E=F$ or $E$ a quadratic \'{e}tale extension of $F$.  Recall that if $\dim(V) \geq 2$, we define
\[d^{V}(s) = \prod_{2 \leq j \leq \dim(V)-2: j \text{ even}}{\zeta(\omega_{\pi},2s-j)}.\]
Note that $d^{V}(s) =1$ if $\dim(V) =2$ or $3$.  When $\dim(V) =1$, we define $d^{V}(s) = \zeta(\omega_{\pi},2s)^{-1}$.  With this definition, one has $d^{V}(s) = \zeta(\omega_{\pi},2s-2)d^{V_1}(s-1)$ if $\dim(V) \geq 3$, and Theorem \ref{GJGSpin} is also true in the case $\dim(V) = 1$.  Thus we may use $V=V_{E}$, $E=F$ or $E$ a quadratic extension of $F$, as our base cases for the induction.
\begin{claim} Theorem \ref{GJGSpin} is true when $V = V_{E}$.\end{claim}
\begin{proof} In this case $\phi_{\alpha} = \alpha$ is an unramified character.  Furthermore, by Lemma \ref{iota0}, $\GSpin(V)(F) = E^\times$, $\Phi$ is the characteristic function of $\mathcal{O}_E$, and $\nu = N_{E/F}$ the norm map if $E$ over $F$ is quadratic, and $\nu(t) = t^2$ if $E = F$.  Thus
\[I(\alpha,s) = \int_{E^\times}{\alpha(t)\charf(t \in \mathcal{O}_E)|N_{E/F}(t)|^{s}\,dt},\]
a Tate local integral.  Comparing with the definition of the $L$-function in subsection \ref{subsec:Ldef} gives the claim in both cases.\end{proof}

\subsection{Case of $\GSp_4$} To make clear the ideas, we first write out a proof of Theorem \ref{GJGSpin} in the case of $\GSpin_5 = \GSp_4$, the subgroup of $g\in \GL_4$ with $g \mm{}{1_2}{-1_2}{}g^{t} = \nu(g) \mm{}{1_2}{-1_2}{}$. Denote by $\tau$ the representation of $M = \GL_1 \times \GL_2$ defined by $\tau = Ind_{B}^{M}(\delta_{P}^{-1/2}\delta_{B}^{1/2} \alpha)$.  We parametrize $M$ by $m(\lambda,y)$, with this pair corresponding to the matrix $m(\lambda,y) = \mm{\lambda y'}{}{}{y}$, where $y' = \det(y) \,^ty^{-1}$.  Applying the Iwasawa decomposition, we have
\begin{align*}I(\alpha,s) &= \int_{M}{\delta_{P}^{-1/2}(m)|\lambda \det(y)|^{s} \charf(y,\lambda y') \phi_{\tau}(m) \left(\int_{U}{\charf(uy)\,du}\right)\,dm} \\ &= \int_{M}{\delta_{P}^{-1/2}(m)|\lambda \det(y)|^{s} \charf(y,\lambda y) \phi_{\tau}(m) ||y||^{-1} |\det(y)|^{-1}\,dm}.\end{align*}
The inner unipotent integral $\int_{U}{\charf(uy)\,du}$ can be computed by observing that it is a bi-$\GL_2(\mathcal{O}_F)$-invariant function of $y$, and then computing it directly for diagonal $y$.  In the general case of $\GSpin(V)$, this integral is computed by Lemma \ref{meas}.  Also, note that $\lambda y \in M_2(\mathcal{O}_F)$ if and only if $\lambda y' \in M_2(\mathcal{O}_F)$, so $\charf(y,\lambda y') = \charf(y,\lambda y)$.

We have $\delta_P(m) = |\lambda|^{3}$.  Also, set 
\[\alpha_1 = \alpha(\mathrm{diag}(p,p,1,1)); \quad \alpha_2 = \alpha(\mathrm{diag}(p,1,p,1)); \quad \alpha_3 = \alpha(\mathrm{diag}(1,p,1,p)); \quad \alpha_4 = \alpha(\mathrm{diag}(1,1,p,p))\]
and denote $\mu(\lambda) = \alpha(m(\lambda,1))$.  

We now integrate over the $\GL_1$ bit of $M$. For $y \in \GL_2$, set $\val(y) = k$ if $y = p^k y_0$ with $y_0 \in M_2(\mathcal{O}_F)\setminus pM_2(\mathcal{O}_F)$, and write $||y|| = |\varpi|^{\val(y)}$.  Integrating over $\lambda$ now yields an inner integral of
\[\int_{\GL_1}{|\lambda|^{s-3/2} \charf(\lambda y) \mu(\lambda)\,d\lambda} = ||y||^{-(s-3/2)}\mu(\varpi)^{-\val(y)}\zeta(\mu,s-3/2).\]
Thus
\begin{align*} I(\alpha,s) &= \zeta(\mu,s-3/2)\int_{\GL_2}{||y||^{-(s-1/2)}\mu(\varpi)^{-\val(y)}|\det(y)|^{s-1}\charf(y)\phi_{\tau}(y)\,dy}\\ &= \zeta(\mu,s-3/2) \int_{\GL_2}{\charf(\val(y)=0)\phi_{\tau}(y) |\det(y)|^{s-1}\,dy} \int_{\GL_1}{|t|^{-s+1/2}\mu(t)^{-1}|t|^{2s-2}\omega_{\pi}(t)\,dt} \\ &= \zeta(\mu,s-3/2)\zeta(\omega_{\pi}/\mu,s-3/2) \int_{\GL_2}{\charf(\val(y)=0)\phi_{\tau}(y) |\det(y)|^{s-1}\,dy} \\ &= \frac{\zeta(\alpha_1,s-3/2) \zeta(\alpha_4,s-3/2)}{\zeta(\omega_{\pi},2s-2)}\int_{\GL_2}{\charf(y)|\det(y)|^{s-1}\phi_{\tau}(y)\,dy}.\end{align*}
But now by Godement-Jacquet theory for $\GL_2 = \GSpin_3$, the $\GL_2$ integral is $\zeta(\alpha_2,s-3/2)\zeta(\alpha_3,s-3/2)$.  Hence
\[I(\alpha,s) = \frac{L(\pi,Spin,s-3/2)}{L(\omega_{\pi},2s-2)},\]
as desired. (The ``Standard" $L$-function on $\GSpin_5$ is the ``Spin" $L$-function of $\GSp_4$.)

\subsection{General case} We now go back to the general case.  Recall $P = MN$ with $N \simeq V_1$ and $M \simeq \GL_1 \times \GSpin(V_1)$.  
We begin the computation of the Satake transform of $\Phi$ by computing an inner integral over $N$.  Recall that for $y$ in $\GSpin(V_1)$, we set
\[U_y = \{ x \in V_1: xy \in \Clif(\Lambda_1)\}.\]
In Lemma \ref{meas}, we computed the measure of $U_y$ to be $||y||^{2-\dim(V_1)} |\nu(y)|^{-1}$.

Now we compute the Satake transform.  Recall that $\phi_\alpha$ is induced from the $K_1$-spherical representation $\tau$ on $M$, and denote by $\mu$ the character of $\GL_1$ that is the restriction of $\tau$ to $m(\lambda,1)$.  Note that $\delta_P(m(\lambda,y)) = |\lambda|^{\dim V_1}$, and $\Phi(m(\lambda,y)) = \charf(y, \lambda y)$.  We obtain
\begin{align*} I(\alpha,s) &= \int_{M}{\delta_P^{-1/2}(m) \phi_{\tau}(m) |\nu(m)|^{s} \charf(y,\lambda y) ||y||^{2-\dim V_1} |\nu(y)|^{-1}\,dm} \\ &= \int_{\GSpin(V_1)}{\phi_{\tau}(m(1,y)) \charf(y) |\nu(y)|^{s-1} ||y||^{2-\dim V_1}\,dy}\int_{\GL_1}{|\lambda|^{s-\dim V_1/2}\mu(\lambda) \charf(\lambda y)\, d\lambda}.\end{align*}
The inner integral is
\[\mu(\varpi)^{-\val(y)}||y||^{\dim V_1/2 -s} \zeta(\mu,s-\dim V_1/2).\]
Thus
\begin{align*} I(\alpha,s) &= \zeta(\mu,s-\dim V_1/2) \int_{\GSpin(V_1)}{\mu(\varpi)^{-\val(y)} \phi_{\tau}(m(1,y)) \charf(y) |\nu(y)|^{s-1} ||y||^{2-\dim V_1/2 -s}\,dy} \\ &= \zeta(\mu,s-\dim V_1/2) \int_{\GSpin(V_1)}{\charf(\val(y) = 0)\phi_{\tau}(y) |\nu(y)|^{s-1}\,dy} \\ &\quad \times \int_{\GL_1}{\charf(t) \mu(t)^{-1}\omega_{\pi}(t) |t|^{2s-2}|t|^{2 - \dim V_1/2 -s}\,dt} \\ &= \zeta(\mu,s-\dim V_1/2)\zeta(\omega_{\pi}/\mu,s-\dim V_1/2)\int_{\GSpin(V_1)}{\charf(\val(y)=0)\phi_{\tau}(y) |\nu(y)|^{s-1}\,dy} \\ &= \frac{\zeta(\mu,s-\dim V_1/2)\zeta(\omega_{\pi}/\mu,s-\dim V_1/2)}{\zeta(\omega_{\pi},2s-2)}\int_{\GSpin(V_1)}{\phi_{\tau}(y)\charf(y) |\nu(y)|^{s-1}\,dy} \\ &= \frac{\zeta(\mu,s-\dim V_1/2)\zeta(\omega_{\pi}/\mu,s-\dim V_1/2)}{\zeta(\omega_{\pi}, 2s-2)} I_{V_1}(\tau,s-1).\end{align*}
Since $d^{V}(s) = \zeta(\omega_{\pi},2s-2)d^{V_1}(s-1)$, Theorem \ref{GJGSpin} follows.

\begin{remark}\label{rmk:exactL} It is easy to remove the denominator $d^V(s)$ from Theorem \ref{GJGSpin}, as follows.  For $g \in \Clif(V)$, define $\Phi_{\geq n}(g) = \Phi(\varpi^{-n} g) = \charf(\val(g) \geq n)$, the characteristic function of those $g \in \Clif(V)$ for which $\val(g) \geq n$, or equivalently, the characteristic function of $\varpi^{n}\Clif(\Lambda)$.  Similarly, define $\Phi_{n}(g) = \Phi_{\geq n}(g) - \Phi_{\geq n+1}(g) = \charf(\val(g) = n)$.  Note that, if $f$ is a function on $\GSpin(V)$, then
\begin{equation}\label{center_trans}\omega_{\pi}(\varpi)^n|\varpi|^{2sn}\int_{\GSpin(V)}{f(g)|\nu(g)|^{s}\langle v_0^\vee, \pi(g) v_0 \rangle \,dg} = \int_{\GSpin(V)}{f(\varpi^{-n} g)|\nu(g)|^{s}\langle v_0^\vee, \pi(g) v_0 \rangle \,dg}.\end{equation}

Now, set $q = |\varpi|^{-1}$, and for $N \geq 0$ define polynomials $p_{V,N}(q)$ by the generating series
\[\prod_{2 \leq j \leq \dim(V)-2: j \equiv 0 (2)}{(1-q^j X)^{-1}} = \sum_{N \geq 0}{p_{V,N}(q)X^N}.\]
Set $\Phi'(g) = \sum_{N \geq 0}{p_{V,N}(q)\Phi_{\geq N}(g)}$.  Then it follows from Theorem \ref{GJGSpin} and (\ref{center_trans}) that
\[L(\pi,Std,s+1-\dim(V)/2) = \int_{\GSpin(V)(F)}{\Phi'(g)|\nu(g)|^{s}\langle v_0^\vee, \pi(g) v_0\rangle \,dg}.\]
Equivalently, $\Phi'(g) = \sum_{M \geq 0}{p'_{V,M}(q)\Phi_{M}(g)}$ with the polynomials $p'_{V,M}$ defined by the generating series
\[\sum_{M \geq 0}{p'_{M,V}(q)X^M} = (1-X)^{-1}\sum_{N \geq 0}{p_{V,N}(q)X^N} = \prod_{0 \leq j \leq \dim(V)-2: j \equiv 0 (2)}{(1-q^j X)^{-1}}.\]

As in \cite{getz}, one can also get rid of the denominator $d^V(s)$ by putting Schwartz-Bruhat functions on auxiliary copies of $\GL_1$.
\end{remark}

\section{The pullback formula}\label{section:pullback}
In this section we give a quick review of the pullback formula, due to Garrett \cite{garrett, garrett2} and Shimura \cite{shimuraBook}, in the special case of orthogonal groups.  Analogous results \cite{garrett2, shimuraBook} hold for symplectic and unitary groups.  The setup is as follows.  We have two quadratic spaces $V_0 \hookrightarrow V$, with $V = U^\vee \oplus V_0 \oplus U$.  Here $U$ is isotropic and $U^\vee \oplus U$ is $V_0^\perp$ inside of $V$. Out of this data, one constructs another quadratic space $W = V \oplus V_0^{-}$, where here $V_0^{-}$ means the space $V_0$ as a vector space but with quadratic form the negative of that of $V_0$.  The space $W$ is split; it has maximal isotropic subspace $X = \Delta(V_0) \oplus U$ with $\Delta(V_0) = \{((0,v,0),v) \in W: v \in V_0\}$.  We write $P_{X}$ for the maximal (Siegel) parabolic subgroup of $\SO(W)$ that stabilizes $X$, and $P_U$ for the parabolic subgroup of $\SO(V)$ that stabilizes the isotropic subspace $U$ of $V$.

Suppose $\alpha$ is an automorphic cusp form on $\SO(V_0^{-})$, and $\Phi$ is a Schwartz-Bruhat function on the spinor module $S_X(W):=(X\Clif(W))\backslash \Clif(W)$.  Here $X\Clif(W) \subseteq \Clif(W)$ is the right ideal generated by $X$.  Denote by $1_X$ the image of $1$ from $\Clif(W)$ in $X\Clif(W)\backslash \Clif(W)$.  Out of $\Phi$ one defines a Siegel Eisenstein series on $\SO(W)$: One sets
\[f_{X}(g,\Phi,s) = |\nu(g)|^{s} \int_{\GL_1(\A)}{\Phi(t 1_{X} g)|t|^{2s}\,dt}\]
and $E_{X}(g,\Phi,s) = \sum_{\gamma \in P_{X}(F)\backslash \GSpin(W)(F)}{f(\gamma g,\Phi,s)}$.  One can pull back the Siegel Eisenstein series $E_{X}(g,\Phi,s)$ to the product $\SO(V) \times \SO(V_0^{-})$, and integrate against the cusp form $\alpha$ on $\SO(V_0^{-})$ to get an automorphic function on $\SO(V)$.  The pullback formula says that this construction yields an Eisenstein series on $\SO(V)$ for the parabolic $P_{U}$ with data given by $\alpha$.

We now explain this more precisely.  First, we give the requisite properties of the function $f_{X}(g,\Phi,s)$.  If $p \in P_{X}$, then the image of $p \in \SO(W)$ is of the form $\mm{m(p)}{*}{}{(m(p)^t)^{-1}}$ for a matrix $m(p) \in \GL(W/X)$.  We have the following lemma.
\begin{lemma}\label{EisPsec} If $p \in P_{X}$, then $1_{X}p = \alpha(p)1_{X}$ in $S_X(W) = (X\Clif(W))\backslash \Clif(W)$ for a character $\alpha: P \rightarrow \GL_1$.  This character $\alpha$ satisfies $\alpha(p)^2/\nu(p) = \det(m(p))^{-1}$.  Consequently, the section $f_{X}(g,\Phi,s)$ satisfies 
\[f_{X}(pg,\Phi,s) = |\det(m(p))|^{s}f_{X}(g,\Phi,s) = \delta_{P_X}(p)^{\frac{s}{n-1}}f_{X}(g,\Phi,s)\]
for all $p \in P_{X}$, where $\delta_{P_X}$ is the modulus character of $P_{X}$ and $n = \dim(X)$. \end{lemma}
\begin{proof} Recall the filtration $W^X_{\bullet}$ on $\Clif(W)$ defined in subsection \ref{subsec:filtration}.  The filtration $W^X$ induces one $\overline{W^X}$ on $S_X(W)$ by saying $s \in \overline{W^X_{\ell}} S_X(W)$ if there exists $\tilde{s} \in W^{X}_\ell \Clif(W)$ whose image in $S_X(W)$ is $s$.  With this definition, one sees that $1_{X}$ spans all of $\overline{W^X_0}$.  It thus follows from Lemma \ref{PUVlemma} that $1_X$ is an eigenvector for all of $P_X$, i.e., that there exists $\alpha: P_X \rightarrow \GL_1$ with $1_X p = \alpha(p) 1_X$ for all $p \in P_X$.

Choose a basis $x_1, x_2, \ldots, x_n$ for $X$, and set $x = x_1 \cdots x_n$ in $\Clif(W)$.  Then on the one hand, if $p \in P_X$, then $p^{-1}x p = \det(m(p))^{-1}x$.  On the other hand, since $p \in \alpha(p) + X\Clif(W)$, $p^{-1}xp = \nu(p)^{-1} p^* xp = \alpha(p)^2\nu^{-1}(p) x$.  The lemma follows.  \end{proof}

It follows from the lemma that $E_{X}(g,\Phi,s)$ is in fact defined and is a Siegel Eisenstein series. Now, the decomposition $W = V \oplus V_0^{-}$ gives rise to maps $\GSpin(V) \rightarrow \GSpin(W)$, $\GSpin(V_0^{-}) \rightarrow \GSpin(W)$ and an injection 
\[(\GSpin(V) \times \GSpin(V_0^{-}))/\{(z,z^{-1}): z \in \GL_1\} \rightarrow \GSpin(W).\]
For $g \in \GSpin(V)$, set
\begin{align*}f_{U}(g,\Phi,\alpha,s) &= |\nu(g)|^{s} \int_{\GSpin(V_0^{-})(\A)}{\Phi(1_{X}(g,h))\alpha(h)|\nu(h)|^{s}\,dh} \\ &= \int_{\SO(V_0^{-})(\A)}{f_{X}((g,h),\Phi,s)\alpha(h)\,dh}.\end{align*}
For $Re(s)$ sufficiently large, these integrals converge absolutely. Note that if $\gamma \in P_U(F)$, then $f_{U}(\gamma g,\Phi,\alpha,s) = f_{U}(g,\Phi,\alpha,s)$.  Furthermore, suppose $p \in P_{U}$, and $p$ acts trivially on $V_0 = U^{\perp}/U$.  Then the image of $p$ in $\SO(V)$ is a matrix of the form 
\[\left(\begin{array}{ccc} m(p) &*&*\\ & 1_{V_0}&*\\ & &(m(p)^{t})^{-1}\end{array}\right)\]
for some element $m(p) \in \GL(W/U^\perp)$, and one has $f_{U}(pg,\Phi,\alpha,s) = |\det(m(p))|^{s}f(g,\Phi,\alpha,s)$.  Thus one can define an Eisenstein series $E_{U}(g,\Phi,\alpha,s) = \sum_{\gamma \in P_U(F) \backslash \GSpin(V)(F)}{f_{U}(\gamma g,\Phi,\alpha,s)}$. 

The following proposition is the pullback formula.
\begin{proposition}[Garrett \cite{garrett, garrett2} Shimura \cite{shimuraBook}] \label{prop:pullback} Suppose $\alpha$ is cuspidal and $V_0$ is anisotropic.  Then one has the identity
\[E_{U}(g,\Phi,\alpha,s) =\int_{\GSpin(V_0^{-})(F)Z(\A)\backslash \GSpin(V_0^{-})(\A)}{E_{X}((g,h),\Phi,s)\alpha(h)\,dh}.\]
Furthermore, assume $\tau$ is a cuspidal automorphic representation of $\SO(V_0)$, and $\alpha$ is a cusp form in the space of $\tau$.  Then the section $f_U(g,\Phi,\alpha,s)$ satisfies the normalization property
\[f_U(1,\Phi,\alpha,s) = \frac{L^S(\tau,Std,s+1-\dim(V_0)/2)}{d^{V_0}(s)}\int_{\GSpin(V_0^{-})(\A_S)}{\Phi_{S}(1_X(1,h))\alpha(h)|\nu(h)|^{s}\,dh}\]
for a sufficiently large finite set of places $S$.  Here $\Phi = \Phi^{S} \otimes \Phi_S$ is a factorization of $\Phi$ into Schwartz-Bruhat functions away from $S$ and at the finite set of places $S$, and $\A_S = \prod_{v \in S}{F_v}$. \end{proposition}

The pullback formula holds more generally than when $V_0$ is anisotropic, but for simplicity we just consider this case.  We begin with the following lemma.
\begin{lemma}\label{pullbackPhi} Suppose $F$ is a nonarchimedean local field.  Denote by $\Lambda(V)$ a maximal lattice in $V$, and similarly $\Lambda(V_0), \Lambda(U), \Lambda(U^\vee)$ lattices in $V_0$, $U$, and $U^\vee$, and assume $\Lambda(V) = \Lambda(U^\vee) \oplus \Lambda(V_0) \oplus \Lambda(U)$.  Set $\Lambda(W) = \Lambda(V) \oplus \Lambda(V_0^{-})$, where $\Lambda(V_0^{-})=\Lambda(V_0)$, except considered inside $V_0^{-}$.  Suppose $\Phi_{X}$ is the characteristic function of the image of $\Clif(\Lambda(W))$ inside $S_{X}(W)$.  Then $\Phi_{X}(1_X(1,h)) = \Phi_{\Lambda(V_0^{-})}(h)$, where $\Phi_{\Lambda(V_0^{-})}$ denotes the characteristic function of $\Clif(\Lambda(V_0^{-}))$ inside $\Clif(V_0^{-})$.  Similarly, $\Phi_{X}(1_X(h,1)) = \Phi_{\Lambda(V_0)}(h)$.\end{lemma}
\begin{proof} Clearly if $\Phi_{\Lambda(V_0^{-})}(h) \neq 0$, then $\Phi_X(1_X(1,h)) \neq 0$.  Conversely, suppose $\Phi_X(1_X(1,h)) \neq 0$.  Then there is $g_0 \in \Clif(\Lambda(W))$ so that $h-g_0 \in X\Clif(W)$.  Since $W = U^\vee \oplus V_0^{-} \oplus X$ and $\Lambda(W) = \Lambda(U^\vee) \oplus \Lambda(V_0^{-}) \oplus \Lambda(X)$, one has $\Clif(\Lambda(W)) = \Clif(\Lambda(U^\vee)) \otimes \Clif(\Lambda(V_0^{-})) \otimes \Clif(\Lambda(X))$.  Hence there is $g_1 \in \Clif(\Lambda(U^\vee)) \otimes \Clif(\Lambda(V_0^{-}))$ so that $h-g_1 \in X\Clif(W)$.  But $\left(\Clif(U^\vee) \otimes \Clif(V_0^{-})\right) \cap \left(X\Clif(W)\right) =0$ inside $\Clif(W)$.  Hence $h-g_1 = 0$, i.e. $h \in \Clif(V_0^{-}) \cap \left(\Clif(\Lambda(U^\vee)) \otimes \Clif(\Lambda(V_0^{-}))\right)$.  It follows that $h \in \Clif(\Lambda(V_0^{-}))$, i.e., that $\Phi_{\Lambda(V_0^{-})}(h) \neq 0$. \end{proof}

\begin{proof}[Proof of Proposition \ref{prop:pullback}] We sketch the proof. Set $H = \SO(V) \times \SO(V_0^{-})$.  First, one computes the double coset $P_X(F) \backslash \SO(W)(F) \slash H(F)$.  When $V_0^{-}$ is anisotropic, this double coset is a singleton, as proved in Lemma \ref{orbitLemma} below.  Furthermore, the stabilizer of the element $1\in \SO(W)(F)$ is the subgroup $P_U^\Delta \in H$, which consists of the elements $(p,h) \in \SO(V) \times \SO(V_0^{-})$ where $p \in P_U$ and $p$ acts on $V_0 = U^\perp/U$ via the element $h \in \SO(V_0)$.  Hence,
\[E_{X}((g,h),\Phi,s) = \sum_{\gamma \in P_U^\Delta(F)\backslash H(F)}{f_X(\gamma (g,h),\Phi,s)}\]
and thus
\begin{align*} \int_{[\SO(V_0^{-})]}{E_{X}((g,h),\Phi,s)\alpha(h)\,dh} &= \int_{[\SO(V_0^{-})]}{\sum_{\gamma \in P_U^\Delta(F)\backslash H(F)}{f_X(\gamma(g,h),\Phi,s)}\alpha(h)\,dh} \\ &= \sum_{\gamma \in P_U(F)\backslash \SO(V)(F)}\int_{\SO(V_0^{-})(\A)}{f_X(\gamma (g,h),\Phi,s)\alpha(h)\,dh} \\ &= \sum_{\gamma \in P_U(F)\backslash \SO(V)(F)}{f_U(\gamma g,\Phi,\alpha,s)} \\ &= E_U(g,\Phi,\alpha,s).\end{align*}
This proves the first part of the proposition. 

The second part of the proposition follows by Lemma \ref{pullbackPhi}, Theorem \ref{GJGSpin}, and the method of ``non-unique functional" integral representations.  This method is explained, for example, in \cite[Theorem 1.2]{bfg}. \end{proof}

\begin{lemma}\label{orbitLemma} Set $H = \SO(V) \times \SO(V_0^{-})$ and suppose $V_0$ is anisotropic.  Then $P_{X}(F) \backslash \SO(W)(F) \slash H(F)$ is a singleton. \end{lemma}
\begin{proof} The map $P_{X}(F)\gamma \mapsto X\gamma$ defines a bijection between the coset space $P_{X}(F)\backslash \SO(W)(F)$ and the maximal isotropic subspaces $Y$ of $W$ for which the dimension of $X/(X \cap Y)$ is even \cite{conrad}.  (The maximal isotropic subspaces $Y$ for which $\dim_F(X/(X \cap Y))$ is odd are $\mathrm{O}(W)(F)$-translates of $X$, but not $\SO(W)(F)$-translates of $X$.)  We thus want to understand the $\SO(V) \times \SO(V_0^{-})$ action on such maximal isotropic subspaces of $W$.  Suppose $Y$ is such a maximal isotropic subspace of $W$.  We will show that there is $\gamma \in H(F)$ moving $Y$ to $X$.  

The key idea is this: Set 
\[pr_V(Y) = \{v \in V: \exists v' \in V_0^{-} \text{ s.t. } (v,v') \in Y\}\]
and similarly define $pr_{V_0^{-}}(Y)$.  Then $Y$ defines a map
\[\alpha_Y: pr_V(Y)/(Y \cap V) \rightarrow pr_{V_0^{-}}(Y)/(Y \cap V_0^{-}).\]
If $v \in pr_V(Y)$ then $\alpha_Y(v)$ is the image in $pr_{V_0^{-}}(Y)/(Y \cap V_0^{-})$ of any $v_0$ for which $(v,v_0)$ is in $Y$.  The fact that $Y$ is isotropic implies $\alpha_Y$ is an isometry.  One also sees easily that $\alpha_Y$ is a linear isomorphism.

Note one also has the exact sequence
\[0 \rightarrow Y \cap V \rightarrow Y \rightarrow pr_{V_0^{-}}(Y)\rightarrow 0.\]
Suppose $\dim Y = \dim(V_0) + \delta$. Then
\[ \dim(V_0) + \delta - \dim(Y \cap V) = \dim pr_{V_0^{-}}(Y) \leq \dim V_0\]
and hence $\dim(Y \cap V) \geq \delta$.  So when $V_0$ anisotropic, we get $\dim(Y \cap V) = \delta$ since $\delta$ is the Witt index of $V$.

Thus, one can move $Y \cap V$ to $U$ via an element of $\SO(V) \times 1 \subseteq H(F)$.  Assume now that $Y \cap V = U$.  Since $\alpha_{Y}$ is an isometry from $V_0$ to $V_0$, it defines an element $\beta \in \mathrm{O}(V_0)$.  Set $\tilde{\beta} = 1 \times \beta$ in $\mathrm{O}(W)$.  Then one obtains that $Y = X \cdot \tilde{\beta}$.  Since by assumption $\dim(X/(X \cap Y))$ is even and not odd, $\tilde{\beta}$ and thus $\beta$ has determinant $1$. Hence one can move the isotropic space $Y$ to $X$ via an element of $H(F)$.  This finishes the proof.\end{proof}

\section{Global integrals}\label{section:globInt} In this section we show how a few local computations that arise in Rankin-Selberg integrals for the standard $L$-function on $\SO(V)$ can be quickly dispensed with by using Theorem \ref{GJGSpin} (proved in section \ref{section:Lfcn}) and the pullback formula (reviewed in section \ref{section:pullback}.)

\subsection{The doubling integral} In this subsection, we reconsider the doubling integral of Piatetski-Shapiro and Rallis \cite{gpsrDoubling}.  Here the situation is as follows.  One has a quadratic space $V$, and a cuspidal representation $\pi$ on $\SO(V)$.  Denote $V^{-}$ the negative of the space of $V$.  That is, $V^{-} = V$ as a vector space, but the quadratic form on $V^{-}$ is the negative of that of $V$.  Now set $W = V \oplus V^{-}$ so that $X = \{(v,v)\}$ is a maximal isotropic subspace of $W$.  Denote by $P_{X}$ the (Siegel) parabolic subgroup of $\SO(W)$ stabilizing $X$.  Suppose $F$ is a nonarchimedean local field for which $V$ is quasisplit.  Then the local integral that must be computed is
\begin{equation}\label{locIntDub}\int_{\SO(V)(F)}{f_X((g,1),s)\langle v_0^\vee, \pi(g) v_0\rangle \,dg},\end{equation}
where $f_X$ is a $K$-spherical section for the induction from $P_{X}$ to $\SO(W)$, and $v_0, v_0^\vee$ are $K$-spherical vectors in $\pi$ and $\pi^\vee$, respectively.

As explained in section \ref{section:pullback}, one can explicitly construct a section $f_{X}((g,1),s)$ using a Schwartz-Bruhat function $\Phi_{X}$ on $S_{X}:=(X\Clif(W))\backslash \Clif(W)$.  In the notation of section \ref{section:pullback}, $U = 0$ and $V_0 = V$.  Using this section, the local integral (\ref{locIntDub}) is
\begin{equation}\label{locIntDub2}\int_{\GSpin(V)(F)}{\Phi_{X}(1_{X}(g,1))|\nu(g)|^{s}\langle v_0^\vee, \pi(g) v_0\rangle \,dg}.\end{equation}
Applying Lemma \ref{pullbackPhi}, it follows that when $\Phi_{X}$ is the characteristic function of the image of $\Clif(\Lambda(W))$ in $S_X(W)$, (\ref{locIntDub2}) is equal to
\[\int_{\GSpin(V)(F)}{|\nu(g)|^{s} \Phi(g)\langle v_0^\vee, \pi(g) v_0\rangle \,dg}.\]
By Theorem \ref{GJGSpin}, this is $\frac{L(\pi,Std,s+1-\dim(V)/2)}{d^V(s)}$, as desired.

\subsection{The period integral} In this subsection we discuss what can be called the ``period" integral for cusp forms on $\SO(V)$, because this global Rankin-Selberg convolution unfolds to a period of the cusp form in question.  The integral uses a cusp form on $\SO(V)$, and a certain Eisenstein series on $\SO(V')$, where $\dim(V') = \dim(V) +1$.  This integral goes back to Murase-Sugano \cite{muraseSugano}, and was vastly generalized by the book \cite{gpsr}.  As will be clear in a moment, this integral is closely connected to the doubling integral above\footnote{That this period integral yields the standard $L$-function of cuspidal representations $\pi$ on $\SO(V)$ essentially follows directly from the doubling integral and the pullback formula.}.  

We now give the details. Let $V$ be an orthogonal space, and choose an $F$-line $T$ in $V$ with $q(T) \neq 0$.  Then $V = V_0 \oplus T$ where $V_0 = T^\perp$.  For simplicity, we assume $V_0$ is anisotropic.  Set $V' = V \oplus T^{-} = (T \oplus T^{-}) \oplus V_0$.  Choose an isotropic basis $e,f$ for $T \oplus T^{-}$, so that $(e,f) =1$.  Then $V' = Fe \oplus V_0 \oplus Ff$.  The Eisenstein series used on $V'$ is associated to the parabolic $P_{f}$ that stabilizes the line $Ff$, and has data on the $\SO(V_0)$ part of the Levi.

More precisely, this Eisenstein series on $\SO(V')$ may be constructed via the pullback formula.  Setting $U=Ff$, then in the notation of section \ref{section:pullback}, the Eisenstein series on $V'$ is $E_U(g,\Phi,\alpha,s)$ for $\alpha$ a cusp form on $\SO(V_0)$ and $\Phi_{X}$ a Schwartz-Bruhat function on $S_{X}= (X\Clif(W))\backslash \Clif(W)$.  Suppose $\pi$ is a cuspidal representation of $\SO(V)$ and $\phi$ is a cusp form in the space of $\pi$.  The global integral is
\begin{align*} I(\phi,\alpha,s) &= \int_{\SO(V)(F) \backslash \SO(V)(\A)}{\phi(g) E_{U}(g,\alpha,\Phi_{X},s)\,dg} \\ &= \int_{[\SO(V) \times \SO(V_0^{-})]}{\phi(g)E_{X}((g,h),\Phi_{X},s)\alpha(h)\,dh\,dg}.\end{align*}
(Again, see section \ref{section:pullback} for the notation.)

We now unfold the integral.  Define $\SO(V_0)^\Delta \simeq \SO(V_0)$ to be the elements $(g,h) \in \SO(V) \times \SO(V_0^{-})$ for which $g$ fixes $T$ and acts on $V_0$ by $h$.  One obtains
\begin{align*} I(\phi,\alpha,s) &= \int_{[\SO(V) \times \SO(V_0^{-})]}{\phi(g)E_{X}((g,h),\Phi_{X},s)\alpha(h)\,dh\,dg} \\ &= \int_{\SO(V_0)^{\Delta}(F)\backslash (\SO(V) \times \SO(V_0^{-}))(\A)}{\phi(g)f_{X}((g,h),\Phi_{X},s)\alpha(h)\,dh\,dg} \\ &= \int_{\SO(V_0)^{\Delta}(F)\backslash (\SO(V) \times \SO(V_0^{-}))(\A)}{\phi(hg)f_{X}((g,1),\Phi_{X},s)\alpha(h)\,dh\,dg} \\ &= \int_{\SO(V)(\A)}{\phi_{\alpha}(g)f_{X}((g,1),\Phi_{X},s)\,dg}.\end{align*}
Here
\[\phi_{\alpha}(g) = \int_{\SO(V_0)(F) \backslash \SO(V_0)(\A)}{\alpha(h)\phi(hg)\,dh}\]
is the period integral of the cusp form $\phi$ mentioned above.  

Using the definition of $f_X(g,\Phi_{X},s)$, one obtains
\[I(\phi,\alpha,s) = \int_{\GSpin(V)(\A)}{|\nu(g)|^{s}\phi_{\alpha}(g) \Phi_{X}(1_{X}(g,1))\,dg}.\]
But $\Phi_{X}(1_{X}(g,1)) = \Phi(g)$ for some Schwartz-Bruhat function $\Phi$ on $\Clif(V)$.  Hence
\[I(\phi,\alpha,s) = \int_{\GSpin(V)(\A)}{|\nu(g)|^{s}\phi_{\alpha}(g) \Phi(g)\,dg}\]
which by Theorem \ref{GJGSpin} yields the $L$-factor $\frac{L(\pi_p,Std,s+1-\dim(V)/2)}{d^V(s)}$ almost everywhere.  Thus the global integral $I(\phi,\alpha,s)$ yields the partial Standard $L$-function of $\phi$. 

\subsection{The Bessel integral} In this subsection we consider what may be called a Bessel integral for cusp forms $\phi$ on $\SO(V)$, because the global Rankin-Selberg convolution unfolds to a Bessel functional of the cusp form in question.  This integral goes back to Sugano \cite{sugano}.  One uses an Eisenstein series on an orthogonal group $\SO(V')$, where $\dim(V') = \dim(V) -1$.

We now give the details.  First, take a quadratic space $V_0$, and an anisotropic vector $T$ in $V_0$. Write $V_1$ for the perpendicular space to $T$ inside $V_0$; we assume $V_1$ is anisotropic.  Set $V = Fe \oplus V_0 \oplus Ff$ and $V' = Fe \oplus V_1 \oplus Fe$, so that $V' \subseteq V$ as $T^\perp$.  Here, as usual, $e,f$ are isotropic vectors with $(e,f) = 1$ and both $e,f$ are perpindicular to $V_0$.  The integral involves an Eisenstein series on $V'$, for the parabolic that stabilizes the line $Ff$, with automorphic data on the Levi $\GL_1 \times \SO(V_1)$.

This Eisenstein series on $\SO(V')$ may be constructed from the pull-back formula.  More precisely, set $W = V' \oplus V_1^{-}$, and let $\alpha$ be cusp form on $\SO(V_1)$. As in section \ref{section:pullback}, let $X \subseteq W$ consist of the elements $\{((x+\lambda f,x): x \in V_1, \lambda \in F\}$, so that $X$ is a maximal isotropic subspace of $W$.  Denote by $\Phi_{X}$ a Schwartz-Bruhat function on $S_{X} = (X\Clif(W))\backslash \Clif(W)$.  The Eisenstein series on $\SO(V')$ is then
\[E_{U}(g,\Phi_{X},\alpha,s) = \int_{[\SO(V_1^{-})]}{E_X((g,h),\Phi_X,s)\alpha(h)\,dh}\]
where here $U = Ff$ and one uses the inclusion $\SO(V') \times \SO(V_1^{-}) \subseteq \SO(W)$.

Suppose $\pi$ is a cuspidal representation of $\SO(V)$, and $\phi$ a cusp form in the space of $\pi$.  The global integral is then
\[I(\phi,s,\alpha) = \int_{[\SO(V')]}{\phi(g)E_{U}(g,\Phi_{X},\alpha,s)\,dg}.\]

Denote by $P'=M'N'$ the parabolic subgroup of $\SO(V')$ or $\GSpin(V')$ that stabilizes the line $Ff$ inside $V'$, and $P=MN$ the parabolic that stabilizes the line $Ff$ inside $\SO(V)$ or $\GSpin(V)$.  The Levi $M'$ is defined to be the one that fixes the decomposition $V' = Fe \oplus V_1 \oplus Ff$ and similarly for $M$.  Recall the maps $n: V \simeq N$ and $m: \GL_1 \times \GSpin(V') \simeq M'$ from subsection \ref{subsec:parabolic}. Define
\[\phi_T(g) = \int_{N(F) \backslash N(\A)}{\psi^{-1}((T,x))\phi(n(x)g)\,dx}\]
and
\[\phi_{T,\alpha}(g) = \int_{\SO(V_1)(F) \backslash \SO(V_1)(\A)}{\alpha(h)\phi_{T}(hg)\,dh}.\]
A straightforward unfolding gives that
\begin{align*} I(\phi,\alpha,s) &= \int_{N'(\A)\backslash \SO(V')(\A)}{\phi_{T,\alpha}(g)f_X((g,1),\Phi_{X},s)\,dg} \\&= \int_{N'(\A)\backslash \GSpin(V')(\A)}{\phi_{T,\alpha}(g)|\nu(g)|^{s}\Phi_{X}(1_{X}(g,1))\,dg}.\end{align*}

Suppose now $F$ is a nonarchimedean local field for which $V$ is quasisplit, and $\pi = \pi_{p}$ is a spherical representation of $\SO(V)(F)$ with space $V_{\pi}$.  Denote by $L: V_{\pi} \rightarrow \C$ a linear functional satisfying $L(n(x) \cdot v) = \psi((T,x))L(v)$ for all $v \in V_{\pi}$, and let $v_0$ be a spherical vector in $V_{\pi}$.  To relate $I(\phi,\alpha,s)$ to the partial $L$-function $L^{S}(\pi,Std,s)$, it suffices to compute the local integrals
\[I_p(L,v_0,s) = \int_{M'(F)}{\delta_{P'}^{-1}(m) |\nu(m)|^{s} \Phi_{X}(1_{X}(m,1))L( m \cdot v_0)\,dm}.\]

Now $M' = \GL_1 \times \GSpin(V_1)$, and we parametrize the elements of $M'$ by $m(\lambda, y)$ as in subsection \ref{subsec:parabolic}.  In order for $L(m(\lambda,y) \cdot v_0)$ to be nonvanishing, one finds $\lambda \in \mathcal{O}_F$.  Applying Lemma \ref{pullbackPhi}, one finds that almost everywhere, the condition $\Phi_X(1_{X}(m,1)) \neq 0$ becomes $y \in \Clif(\Lambda')$, where $\Lambda' \subseteq V'$ is the $\varpi$-adic completion of a fixed global lattice.  Finally, $\delta_{P'}(m(\lambda,y)) = |\lambda|^{\dim V_1}$.  Hence almost everywhere, one must compute the integrals
\[I_p(L,v_0,s) = \int_{\GL_1(F) \times\GSpin(V_1)}{|\lambda|^{-\dim(V_1)}|\lambda \nu(y)|^{s} \charf(\lambda, y) L( m(\lambda,y) \cdot v_0)\,d\lambda\,dy}.\]

In order to understand $I_p(L,v_0,s)$ in terms of $L(\pi,Std,s)$, we will relate the integral $I_p$ to an integral over $M = \GL_1 \times \GSpin(V_0) \subseteq \GSpin(V)$.  To do this, we use the following lemma.
\begin{lemma}\label{lem:redToKS} Assume $\dim V \geq 3$, and suppose $v \in \Lambda_{V_0}$ and $\varpi \nmid 2q(v)$.  Then if $v' \in \Lambda_{V_0}$ and $q(v) = q(v')$, there is an element of $K(\Lambda_{V_0}) = \SO(\Lambda_{V_0})$ that moves $v$ to $v'$. \end{lemma}
\begin{proof} One simply applies (the proof of) Witt's theorem. \end{proof}
It follows from the lemma that almost everywhere, one has an equality of sets
\begin{equation}\label{eqSets}\{ y \in \GSpin(V_0): y \in \Clif(\Lambda_{V_0}), T \cdot y \in \Lambda_{V_0}\} = \{y \in \GSpin(V_1): y \in \Clif(\Lambda_{V_1})\}K(\Lambda_{V_0}).\end{equation}
(One applies Lemma \ref{lem:redToKS} to the elements $T \cdot y$ and $T$.) Finally, note that
\[\delta_P(m(\lambda,y))^{-1}|\nu(y)|^{-1}|\lambda \nu(y)|^{s+1} = |\lambda|^{-\dim(V_1)}|\lambda\nu(y)|^{s}.\]
Hence (almost everywhere) $I_p(L,v_0,s)$ is equal to
\begin{equation}\label{IpBessel}\int_{\GL_1 \times \GSpin(V_0)}{\charf(\lambda, y, T \cdot y)|\nu(y)|^{-1}\delta_{P}^{-1}(m(\lambda,y))|\lambda \nu(y)|^{s+1} L(m(\lambda,y) \cdot v_0)\,d\lambda\,dy}.\end{equation}
In the next subsection, we will relate this last integral to $L(\pi,Std,s)$.

\subsection{The Kohnen-Skoruppa integral} In this subsection we describe a generalization of the integral \cite{ks} (see also \cite{pollackShahKS}) of Kohnen-Skoruppa.  The integral in \cite{ks} and \cite{pollackShahKS} is on $\GSp_4 = \GSpin_5$.  Gritsenko in \cite{gritsenko} made an analogous construction on $\GU(2,2)$, which is essentially quasisplit $\GSpin_{6}$.  We give an integral on $\GSpin(V)$ with $\dim(V) \geq 5$ that extends these constructions\footnote{Just as the period and doubling integrals are closely connected, so too are the Bessel and Kohnen-Skoruppa integrals.}.  When $\dim(V) = 5$, the proof we give essentially reduces to that in \cite{pollackShahKS}, with the caveat that below we make simplifying assumptions that are not made\footnote{Namely, we make the assumption $q(T) \notin q_0(V_0)$ explained below, and treat $p=2$ as a ``bad'' prime.} in \cite{pollackShahKS}.  The reader looking to understand a more complete theory of this integral when $\dim(V) = 5$, including Archimedean calculations, should see \cite{pollackShahKS}.

We now describe the global Rankin-Selberg convolution.  Suppose $(V_0,q_0)$ is an anisotropic quadratic space over the number field $F$, and $U$ is $2$-dimensional $F$ vector space.  We set $V = U^\vee \oplus V_0 \oplus U$ with the quadratic form $q$ defined by $q(\alpha, v_0,\delta) = \alpha(\delta) + q_0(v_0)$.  Fix a basis $e_1, e_2$ of $U^\vee$ and the dual basis $f_1, f_2$ of $U$.  The assumption that $V_0$ is anisotropic is made for simplicity.  The global integral will involve three functions on $\SO(V)$: a cusp form, an Eisenstein series for the parabolic $P_{U}$ stabilizing $U$, and a special function $P_{T}^{\beta}$.

The Eisenstein series used is constructed via the pullback formula.  In the notation of section \ref{section:pullback}, it is $E_U(g,\Phi_{X},\alpha,s)$ for a cusp form $\alpha$ on $\SO(V_0)$.  Set $V_1 = Fe_2 \oplus V_0 \oplus Ff_2$.  To define the special function, choose an auxiliary nonzero vector $T$ in $V_1$, and for simplicity assume that $q(T)$ is not in $q_0(V_0) := \{q(v_0): v_0 \in V_0\}$; in particular, $T$ is not isotropic.  Furthermore, choose a Schwartz-Bruhat function $\beta$ on $V(\A)$.  Then
\[P_{T}^\beta(g) := \sum_{\gamma \in \Stab(T)(F) \backslash \SO(V)(F)}{\beta(T \gamma g)}\]
where $\Stab(T)$ denotes the stabilizer of $T$ in $\SO(V)$.  Suppose $\pi$ is a cuspidal representation on $\SO(V)$ and $\phi$ is a cusp form in the space of $\pi$.  The global integral is
\[I(\phi,\Phi_{X},\alpha,s) = \int_{[\SO(V)]}{\phi(g)E_{U}(g,\Phi_{X},\alpha,s)P_{T}^\beta(g)\,dg}.\]
The integral does depend on the data $T,\beta$, but we drop these from the notation.  By the pullback formula, the global integral is also
\begin{equation}\label{intDef}I(\phi,\Phi_{X},\alpha, s) = \int_{[\SO(V)]\times [\SO(V_0^{-})]}{P_T^{\beta}(g)\phi(g_1)E_{X}((g_1,g_2),\Phi_{X},s)\alpha(g_2)\,dg_1\,dg_2}.\end{equation}

\subsubsection{Coset decompositions} To unfold the integral $I(\phi,\alpha,\Phi_{X},s)$, we prepare by making some coset decompositions. Set $H = \SO(V) \times \SO(V_0^{-})$. It was proved in Lemma \ref{orbitLemma} that $P_{X}(F) \backslash \SO(W)(F) \slash H(F)$ is a singleton.  This enables one to easily unfold the Eisenstein series $E_{X}(g,\Phi_{X},\alpha,s)$.  To unfold the function $P_{T}^\beta(g)$, one requires an additional orbit calculation, which we now give.

\begin{lemma}\label{PTorbit} Suppose $T \in V_1$ is such that $q(T)$ is not in $q(V_0)$.  One has the following.
\begin{enumerate}
\item \label{PT:item1} The subspace of $V_1$ orthogonal to $T$ is anisotropic.
\item \label{PT:item2} The stabilizer $\mathrm{Stab}(T)$ acts transitively on the two-dimensional isotropic spaces in $V = F \oplus V_1 \oplus F$.\end{enumerate}
 \end{lemma}
\begin{proof} We first prove part (\ref{PT:item1}).  We have $V_1 = F \oplus V_0 \oplus F$.  In coordinates, write $T = (\alpha,h,\delta)$, i.e., $T = \alpha e_2 + h + \delta f_2$ with $h \in V_0$.  Since $q(T) \notin q(V_0)$, $\alpha$ and $\delta$ are nonzero.  Suppose $v' = (a',b',c') \in V_1$ is isotropic and nonzero.  Since $V_0$ is anisotropic, $a'c' \neq 0$.  It follows that there exists $\lambda \in F^\times$ so that $T + \lambda v' = (0,*,*)$.  Now suppose $(T,v') = 0$.  Then $q(T) = q(T+\lambda v') = q((0,*,*)) \in q(V_0)$, contradicting the assumption on $T$.  Thus the subspace of $V_1$ orthogonal to $T$ is indeed anisotropic.

Now consider part (\ref{PT:item2}).  Suppose $u_1, u_2$ span an isotropic two-dimensional space in $V$.  By part (\ref{PT:item1}), one checks that we cannot have both $u_i$ orthogonal to $T$.  Hence we may assume $(u_2,T) \neq 0$, and then by linearity that $(u_1,T) = 0$.  Since $(u_1,T) = 0$, one may move $u_1$ to $f_1$ by $\mathrm{Stab}(T)$. (This is because $\mathrm{Stab}(T)$ acts transitively on the isotropic lines in $T^\perp \subseteq V$.)  Now by linearity and the fact that $(u_1, u_2) = 0$, we may assume $u_2 \in V_1$. Scaling $u_2$, we may write $u_2 = T + y'$, with $(y',T) = 0$ and $q(y') = -q(T)$.  Since $q(y')= -q(T)$ is determined, there is a single orbit under $\mathrm{Stab}(T)$.\end{proof}

The next thing to do is to compute $\mathrm{Stab}(T) \cap P_U$.  This stabilizer is computed in the following lemma.  Suppose in coordinates $T=(\alpha,h,\delta)$, i.e. $T = \alpha e_2 + h + \delta f_2$ with $h \in V_0$. Recall that under our assumptions, $\alpha \delta \neq 0$.  Since $(f_1,T) = 0$ but $(f_2,T) \neq 0$, if $g \in \mathrm{Stab}(T) \cap P_U$, then $g$ stabilizes the line spanned by $f_1$, i.e., $g \in P=P_{F f_1}$.  

Write $P = MN$ for the Levi decomposition of $P$, and recall the maps $n: V_1 \simeq N$ and $m: \GL_1 \times \SO(V_1) \simeq M$.
\begin{lemma} Inside $\SO(V)$, $\mathrm{Stab}(T) \cap P_U = \SO(V_0)' m((\lambda,1))n(x)$ where $(x,T) = 0, \lambda \in \GL_1,$ and $\SO(V_0)' \simeq \SO(V_0)$.  More precisely, $\SO(V_0)'$ consists of the elements
\[\SO(V_0)' = \{ m_0 n_{V_0}(x_0): m_0 \in \SO(V_0) \text{ and } h \cdot m_0 + \alpha x_0 = h\}.\]
That is, $x_0 = h \cdot (1-m_0)/\alpha$ is uniquely determined by $m_0$.  Recall here that $T= (\alpha, h, \delta)$. \end{lemma}
\begin{proof} As explained above, $\mathrm{Stab}(T) \cap P_U \subseteq \mathrm{Stab}(T) \cap P$.  Using this, one concludes the lemma by a straightforward computation.\end{proof}

\subsubsection{Unfolding} With the above orbit and stabilizer calculations, we can unfold the global integral $I(\phi,\Phi_{X},\alpha,s)$.
Recall $H = \SO(V) \times \SO(V_0)$.  Define
\[\phi_T(g) = \int_{V_1(F)\backslash V_1(\A)}{\psi^{-1}((T,x))\phi(n(x)g)\,dx}\]
and
\[\phi_{T,\alpha}(g) = \int_{\SO(V_0)'(F)\backslash \SO(V_0)'(\A)}{\alpha(h)\phi_T(hg)\,dh}.\]
Finally, set $N_T \subseteq N$ to be $n\left(\{x \in V_1: (x,T) = 0\}\right)$, and define
\[\phi_{N_T}(g) = \int_{N_T(F)\backslash N_T(\A)}{\phi(n(x)g)\,dx}.\]

One has
\begin{align*} I(\phi,\Phi_{X},\alpha, s) &= \int_{H(F) \backslash H(\A)}{P_T^\beta(g_1) \phi(g_1) E_X(g,s,\Phi_X) \alpha(g_2) \, dg} \\ &= \int_{\left((\mathrm{Stab}(T) \cap P_U) \times pr(*)\right)(F) \backslash H(\A)}{\beta(T g_1) \phi(g_1) f_X(g,s,\Phi_{X}) \alpha(g_2) \, dg} \\ &= \int_{(\GL_1(F) \SO(V_0)'(F)) \times pr(*)(F)N_T(\A) \backslash H(\A)}{\beta(T g_1) \phi_{N_T}(g_1) f_X(g,s,\Phi_{X}) \alpha(g_2) \, dg} \\ &= \int_{(\SO(V_0)'(F)) \times pr(*)(F)N_T(\A) \backslash H(\A)}{\beta(T g_1) \phi_{T}(g_1) f_X(g,s,\Phi_{X}) \alpha(g_2) \, dg} \\ &= \int_{\Delta(\SO(V_0))(\A)N_T(\A) \backslash H(\A)}{\beta(T g_1) \phi_{T,\alpha}(g_1) f_X(g,s,\Phi_{X})\, dg}\end{align*}
where the $pr(*)$ projects $\SO(V_0)'$ to $\SO(V_0)$ and $\Delta(\SO(V_0)) = \{(h',h) \in \SO(V_0)' \times \SO(V_0)\}$.

Hence,
\begin{align}\nonumber I(\phi,\Phi_{X},\alpha,s) &= \int_{N_T(\A)\backslash \SO(V)(\A)}{\beta(T g) \phi_{T,\alpha}(g) f_{X}((g,1),s,\Phi_X)\,dg} \\ &= \label{unfolded} \int_{N_T(\A)\backslash \GSpin(V)(\A)}{\beta(T g) \phi_{T,\alpha}(g) |\nu(g)|^{s} \Phi_{X}(1_{X} \cdot (g,1))\,dg}.\end{align}

\subsubsection{Unramified computation}
We now give the unramified computation corresponding to the unfolded integral (\ref{unfolded}). Suppose then that $F$ is a nonarchimedean local field for which $V$ is quasisplit, and $\pi$ is a spherical representation of $\SO(V)(F)$ with space $V_{\pi}$.  As in the case of the Bessel integral, denote by $L: V_{\pi} \rightarrow \C$ a linear functional satisfying $L(n(x) \cdot v) = \psi((T,x))L(v)$ for all $v \in V_{\pi}$, and let $v_0$ be a spherical vector in $V_{\pi}$.  The local integral that must be computed is
\begin{equation}\label{unramL} I(L,v_0,s) = \int_{N_T(F)\backslash \GSpin(V)(F)}{\beta(Tg)L(g \cdot v_0) |\nu(g)|^{s}\Phi_{X}(1_X \cdot (g,1))\,dg}\end{equation}
where $\beta$ and $\Phi_{X}$ are now assumed to be the characteristic functions of fixed lattices.

 We first simplify the integral (\ref{unramL}), and then relate this simplification to $L(\pi,Std,s)$ by an application of Lemma \ref{FClemma}.

The first step is to compute the integral
\[\beta_T(g) :=\int_{N_T(F)\backslash N(F)}{\psi((T,x))\beta(T n(x)g)\,dx}.\]
Denote by $\Lambda_1$ the fixed maximal lattice in $V_1$.
\begin{lemma} If $m = m(\lambda,y)$ then $\beta_T(m) = |\lambda| \charf( \lambda \in \mathcal{O}_F)\charf(T \cdot y \in \Lambda_1)$. \end{lemma}
\begin{proof} The first step is to evaluate $Tn(x)m$, where $m \in M \subseteq P$, the parabolic stabilizing the line spanned by $f_1$.  If
\[m = \left(\begin{array}{ccc} \lambda & & \\ & y & \\ & & \lambda^{-1} \end{array}\right)\]
in block form via the decomposition $V = Fe_1 \oplus V_1 \oplus F f_1$, then $T n(x) m = (0, T \cdot y, -(T,x)/\lambda)$.  Hence, 
\[\beta_T(m) = |\lambda| \charf( \lambda \in \mathcal{O}_F)\charf(T \cdot y \in \Lambda_1),\]
completing the computation. \end{proof}

Note that the above $m$ is the image of $m(\lambda, y)$ from $\GSpin(V)$. We conclude
\begin{align*} I(L,v_0,s) &= \int_{N_T(F)\backslash \GSpin(V)(F)}{\beta(Tg)L( g\cdot v_0)|\nu(g)|^{s}\Phi_{X}(1_{X} \cdot (g,1))\,dg} \\ &=\int_{N(F)\backslash \GSpin(V)(F)}{\beta_{T}(g)L(g \cdot v_0)|\nu(g)|^{s}\Phi_X(1_{X} \cdot (g,1))\,dg} \\ &= \int_{M(F)}{\delta_{P}^{-1}(m) |\lambda| \charf(\lambda, T \cdot y) L(m \cdot v_0) |\nu(m)|^{s}\Phi_X(1_{X} \cdot (m,1))\,dm} \\&= \int_{M(F)}{\delta_{P}^{-1}(m) |\nu(y)|^{-1} \charf(\lambda, T \cdot y) L(m\cdot v_0) |\nu(m)|^{s+1}\Phi_X(1_{X} \cdot (m,1))\,dm}. \end{align*}
Here, $m = m(\lambda, y)$.

Next, we evaluate $\Phi_X(1_{X} \cdot m(\lambda,y))$. 
\begin{lemma}\label{charyLem} At almost all finite places, one has $\charf(T \cdot y)\Phi_X(1_{X} \cdot m(\lambda,y))= \charf(y, T \cdot y)$. \end{lemma}
\begin{proof} First, note that
\[m(\lambda,y) = (e_1f_1 + \lambda f_1e_1)y = (e_1f_1 +f_1e_1 -f_1e_1 + \lambda f_1e_1) y = y + f_1(\lambda e_1 -e_1)y.\]
Since $f_1 \in X$, the image of $1_{X} \cdot m(\lambda,y)$ in the spinor module $S_X$ is the same as the image of $y$.  Now by (\ref{eqSets}), in the presence of the condition $\charf(T \cdot y)$, we may assume $y$ is in $\GSpin$ of the subspace of $V_1$ orthogonal to $T$.  The lemma now follows by an application of Lemma \ref{pullbackPhi}. \end{proof} 

Thus applying Lemma \ref{charyLem}, we get
\begin{equation} \label{Isimplified} I(L,v_0,s-1) = \int_{M(F)}{\delta_{P}^{-1}(m) |\nu(y)|^{-1} \charf(\lambda, T \cdot y, y) L(m\cdot v_0) |\nu(m)|^{s}\,dm}.\end{equation}
Comparing to (\ref{IpBessel}), observe that this is the same local integral that arose from the Bessel integral, so our manipulations below will also complete the unramified computation for that integral.  
 
For later use, we note that
\begin{equation}\label{ImodZ} I(L,v_0,s-1) = \zeta(\omega_{\pi},2s-2)\int_{M(F)}{\delta_{P}^{-1}(m) |\nu(y)|^{-1} \charf(\lambda, T \cdot y,\val(y)=0) L(m\cdot v_0) |\nu(m)|^{s}\,dm}\end{equation}
by integrating out the center.

\subsubsection{$\Phi$ Fourier coefficient} To complete the unramified computation, we will apply Lemma \ref{FClemma}.  One has
\begin{align*} L(v_0) \frac{L(\pi,Std,s+1-\dim(V)/2)}{d^V(s)} &= L(v_0) \int_{\GSpin(V)(F)}{\Phi(g)|\nu(g)|^{s} \langle v_0^\vee, \pi(g) v_0\rangle \,dg} \\ &= \int_{\GSpin(V)(F)}{\Phi(g)|\nu(g)|^{s} L(g \cdot v_0) \,dg} \\ &= \int_{M(F)}{\delta_{P}^{-1}(m)|\nu(m)|^{s} S_T(\Phi)(m) L(m \cdot v_0)\,dm}. \end{align*}
Applying Lemma \ref{FClemma}, this becomes
\[\int_{M(F)}{\delta_{P}^{-1}(m)|\nu(y)|^{-1}|\nu(m)|^{s} \charf(\lambda,\val(y)=0, T\cdot y) L(m \cdot v_0)\,dm}.\]
Comparing to (\ref{ImodZ}), we conclude
\begin{align*}I(L,v_0,s-1) &= L(v_0) \zeta(\omega_{\pi},2s-2) \frac{L(\pi,Std,s+1-\dim(V)/2)}{d^V(s)} \\ &= L(v_0) \frac{L(\pi,Std,s+1-\dim(V)/2)}{d^{V_1}(s)}.\end{align*}

\bibliography{GJ_bib} 
\bibliographystyle{amsalpha}
\end{document}